\documentclass[12pt,reqno]{amsart}
\usepackage{amsmath,amssymb,amscd,verbatim,color}
\usepackage{graphicx} 
\usepackage{pifont}
\makeatletter
\@namedef{subjclassname@2020}{%
	\textup{2020} Mathematics Subject Classification}
\makeatother
\usepackage{amsfonts}
\usepackage[top=35mm, bottom=35mm, left=30mm, right=30mm]{geometry}

\usepackage{xcolor}
\theoremstyle{plain}
\newtheorem{thm}{Theorem}[section]
\newtheorem{cor}[thm]{Corollary}

\newtheorem{lem}[thm]{Lemma}
\newtheorem{prop}[thm]{Proposition}

\theoremstyle{definition}

\newtheorem{claim}{Claim}

\newtheorem{case}{Case}

\numberwithin{equation}{section}

\newcommand{\mb}{\mathbb}
\newcommand{\mc}{\mathcal}

\allowdisplaybreaks

\begin{document}
	
	\title[ ]{Directional bounded complexity, equicontinuity and discrete spectrum for $\mathbb{Z}^q$-actions}

	\author[C. Liu and L. Xu]{Chunlin Liu and Leiye Xu}
	\address{C. Liu: CAS Wu Wen-Tsun Key Laboratory of Mathematics, School of Mathematical Sciences, University of Science and Technology of China, Hefei, Anhui, 230026, PR China}
	
	\email{lcl666@mail.ustc.edu.cn}
	
	\address{L. Xu: CAS Wu Wen-Tsun Key Laboratory of Mathematics, School of Mathematical Sciences, University of Science and Technology of China, Hefei, Anhui, 230026, PR China}

	\email{leoasa@mail.ustc.edu.cn}

	
	\begin{abstract}
Given $q\in\mathbb{N}$, let $(X,T)$ be a $\mathbb{Z}^q$-system, $\vec{v}\in\mathbb{R}^q\setminus\{\vec{0}\}$ be a direction vector and $\textbf{b}\in\mathbb{R}_+^{q-1}$. We study $(X,T)$ that has bounded complexity with respect to three kinds of metrics defined along direction $\vec{v}$: the directional Bowen metric $d_k^{\vec{v},\textbf{b}}$, the directional max-mean metric $\hat{d}_k^{\vec{v},\textbf{b}}$ and the directional mean metric $\bar{d}_k^{\vec{v},\textbf{b}}$. It is shown that  $(X,T)$ has bounded topological complexity with respect to $\{d_k^{\vec{v},\textbf{b}}\}_{k=1}^{\infty}$ (resp. $\{\hat{d}_k^{\vec{v},\textbf{b}}\}_{k=1}^{\infty}$) if and only if $T$ is $(\vec{v},\textbf{b})$-equicontinuous (resp. $(\vec{v},\textbf{b})$-equicontinuous in the  mean). Meanwhile, it turns out that an invariant Borel probability measure $\mu$ on $X$ has bounded complexity with respect to $\{d_k^{\vec{v},\textbf{b}}\}_{k=1}^{\infty}$ if and only if $T$ is $(\mu,\vec{v},\textbf{b})$-equicontinuous. Moreover, it is shown that $\mu$ has bounded complexity with respect to $\{\bar{d}_k^{\vec{v},\textbf{b}}\}_{k=1}^{\infty}$ if and only if $\mu$ has bounded complexity with respect to $\{\hat{d}_k^{\vec{v},\textbf{b}}\}_{k=1}^{\infty}$  if and only if $T$ is $(\mu,\vec{v},\textbf{b})$-mean equicontinuous if and only if $T$ is $(\mu,\vec{v},\textbf{b})$-equicontinuous in the mean if and only if $\mu$ has $\vec{v}$-discrete spectrum. 
   \end{abstract}
	\maketitle	
	\section{Introduction}
Given $q\in \mathbb{N}$, throughout this paper, we call a pair $(X,T)$  a $\mathbb{Z}^q$-topological dynamical system ($\mathbb{Z}^q$-t.d.s. for short) if $X$ is a compact metric space and the $\mathbb{Z}^q$-action $T:X\to X$ is a homeomorphism from the additive group $\mathbb{Z}^q$ to the group of homeomorphisms of $X$. For a $\mathbb{Z}^q$-t.d.s. $(X,T)$, we denote the corresponding homeomorphism by $T^{\vec{v}}$ for any $\vec{v}\in \mathbb{Z}^q$ so that $T^{\vec{v}}\circ T^{\vec{w}}=T^{\vec{v}+\vec{w}}$ for any $\vec{v},\vec{w}\in \mathbb{Z}^q $ and $T^{\vec{0}} $ is the identity on $X$.  Let $\mathcal{B}_X$ be the Borel $\sigma$-algebra of $X$ and $\mu$ be a Borel probability measure on $(X,\mathcal{B}_X)$. We say that $\mu$ is  invariant for $(X,T)$ if for any $\vec{v}\in \mathbb{Z}^q$ and $A\in\mathcal{B}_X$, $\mu (T^{-\vec{v}}A)=\mu(A)$. The set of invariant Borel probability measures defined on $\mathcal{B}_X$ is denoted by $M(X,T)$. It is well known that for each $\mathbb{Z}^q$-t.d.s. $(X,T)$, $M(X,T)\neq \emptyset$ and each $\mu\in M(X,T)$ induces a $\mathbb{Z}^q$-meausre preserving dynamical system ($\mathbb{Z}^q$-m.p.s. for short) $(X,\mathcal{B}_X,\mu,T)$.  

Since entropy is introduced  to dynamical systems by Kolmogorov \cite{Ko}, which was used to measure the chaoticity or unpredictability of a given system. It is well known that zero entropy
systems is a dense $G_{\delta}$ subset of all homeomorphisms. Thus it is important to study the complexity in zero entropy systems.  For $\mathbb{Z}$-systems, related work can be traced back to Ferenczi \cite{F}, he studied the measure-theoretic complexity of ergodic systems, using the $\alpha$-names of a partition and the Hamming distance. He proved that, when the measure is ergodic, the complexity function is bounded if and only if the system has discrete spectrum. Yu \cite{Y} showed that the result holds for the non-ergodic case. For the topological version,  Blanchard, Host and Maass \cite{B} studied topological complexity via the complexity function of an open cover and showed that the complexity function is bounded for any open cover if and only if the system is equicontiuous.
Recently, in the investigation of Sarnak's conjecture, Huang, Wang and Ye \cite{HWY} introduced the measure complexity of an invariant measure for $\mathbb{Z}$-t.d.s. $(X,T)$, similar to the one introduced by Katok \cite{K}, by using the mean metric instead of the Bowen metric. Moreover, Huang et al. \cite{H1} studied topological and measure-theoretic complexity via a sequence of metrics and  obtained the relation between bounded complexity, mean equicontinuity and discrete spectrum. Combining results in \cite{H1} and \cite{HWY}, we know that a $T$-invariant measure $\mu$ has bounded complexity with respect to mean metric if and only if $T$ is $\mu$-mean equicontinuous and if and only if $\mu$ has discrete spectrum. After that, the above results are extended to amenable group actions \cite{YZZ}.

For more general group actions, such as $\mathbb{Z}^q$-actions, the dynamics become more
and more complicated. To better understand the complexities of zero entropy $\mathbb{Z}^q$-systems, scholars began to consider the directional dynamical systems. In order to study cellular automaton map together with the Bernoulli shift, Milnor \cite{Mil} introduced directional entropy, which was further studied \cite{Bro,CK,P1,P}. 
A natural question is what are the dynamics of their non-cocompact subgroup actions. In this respect, Johnson and \c Sahin \cite{JS} studied directional recurrence. 
More recently, authors in this paper \cite{LX} introduced directional sequence entropy and directional discrete spectrum for $\mathbb{Z}^q$-m.p.s., and showed that a $\mathbb{Z}^q$-m.p.s. has directional discrete spectrum along $q$ linearly independent directions if and only if it has discrete spectrum. In this paper, we further investigate the complexity along any direction and obtain the relation with directional discrete spectrum.

Motivated by the above-mentioned discussions, in this paper,  we introduce topological and measure-theoretic complexity via a seqeuence of metrics induced by a metric $d$ along some direction, and introduce the notion of directional equicontinuity. Meanwhilie, we establish the relation between directional discrete spectrum, directional complexity and directional equicontinuity.
 One of our main results is to show  that for a $\mathbb{Z}^q$-t.d.s. $(X,T)$ with $\mu\in M(X,T)$, $\mu$ has directional discrete spectrum if and only if $\mu$ has bounded complexity with respect to directional metrics. The proof of this result for $\mathbb{Z}$-actions strongly depends on the ergodicity. Howver, the ergodicity along some direction of $\mu$ is not well defined which is our major difficulty.  
 \begin{figure}[!h]
 	\centering
 	\includegraphics[scale=0.6]{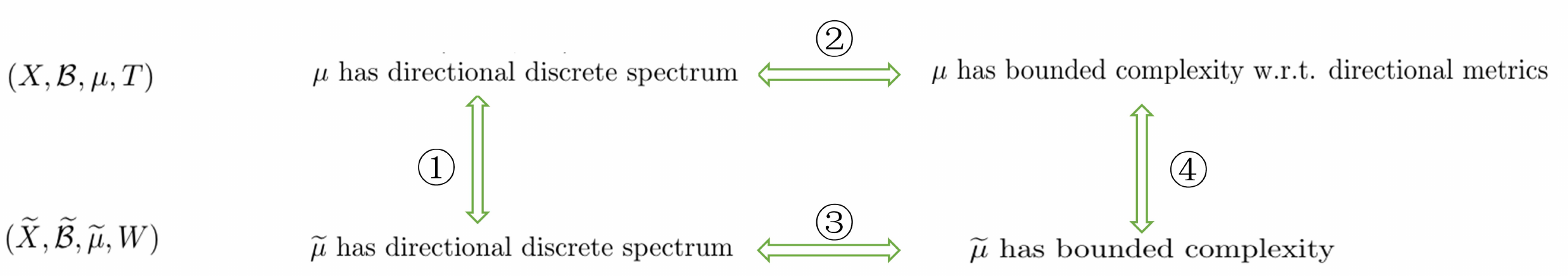}
 	\caption{~}
 	\label{fig:1}
 \end{figure} 
With help of a $\mb{Z}$-m.p.s. $(\widetilde{X},\widetilde{\mc{B}},\widetilde{\mu},W)$, which was introduced by Park \cite{P}, we overcome this difficulty by the following idea (see figure 1): 
The equivalence \textcircled{1} has been shown in \cite{LX}  and the equivalence \textcircled{3} has been obtained in \cite{H1,HWY}. Thus, to prove the equivalence \textcircled{2}, we only need to prove the equivalence \textcircled{4}, which avoids the discussion of ergodicity.

To be precise, let $(X,T)$ be a $\mathbb{Z}^2$-t.d.s.,   $\vec{v}=(1,\beta)\in\mathbb{R}^2$ be a direction vector and $b\in (0,\infty)$. We remark that notations above are adopted for simplicity but all results in this paper hold for $\mathbb{Z}^q$-system with any direction vector $\vec{v}\in\mathbb{R}^q\setminus\{\vec{0}\}$ for any $q\in\mathbb{N}$.  Put $$\Lambda^{\vec{v}}(b)=\left\{(m,n)\in\mathbb{Z}^2:\beta m-b\leq n\leq \beta m+b\right\}$$
and  for each $k\in\mathbb{N}$, let $\Lambda_k^{\vec{v}}(b)=\Lambda^{\vec{v}}(b)\cap ([0,k-1]\times\mathbb{Z})$, where $[0,k-1]=\{0,1,\ldots,k-1\}$. We define three kinds of metrics on $X$ along $\vec{v}$ as follows.  
For $k\in \mathbb{N}$ and $x,y\in X$, let
 $$d_k^{\vec{v},b}(x,y)=\max_{(m,n)\in \Lambda_k^{\vec{v}}(b) }\left\{d(T^{(m,n)}x,T^{(m,n)}y)\right\},$$

$$\hat{d}_k^{\vec{v},b}(x,y)=\max_{1\leq i\leq k}\left\{\frac{1}{\#(\Lambda_i^{\vec{v}}(b))}\sum_{(m,n)\in\Lambda_i^{\vec{v}}(b)}d(T^{(m,n)}x,T^{(m,n)}y)\right\}$$
and
$$\bar{d}_k^{\vec{v},b}(x,y)=\frac{1}{\#(\Lambda_k^{\vec{v}}(b))}\sum_{(m,n)\in\Lambda_k^{\vec{v}}(b)}d(T^{(m,n)}x,T^{(m,n)}y),$$
where  $\#(A)$ is the number of elements of a finite set $A$.
It is clear that, for all $k\in \mathbb{N}$ and $x,y\in X$,
$$d_k^{\vec{v},b}(x,y)\geq \hat{d}_k^{\vec{v},b}(x,y)\geq \bar{d}_k^{\vec{v},b}(x,y).$$

 For $x\in X$, $\epsilon>0$ and a metric $\rho$ on $X$, let $B_{\rho}(x,\epsilon)=\{y\in X: \rho(x,y)<\epsilon\}.$ Let $\{\rho_k\}_{k=1}^{\infty}$ be a sequence of metrics (denoted as $\{\rho_k\}$ for convenience). We say that $(X,T)$ has bounded topological complexity with respect to $\{\rho_k\}$ if for any $\epsilon>0$, there exists a positive integer $C=C(\epsilon)$ such that for each $k\in \mathbb{N}$, there are points $x_1,x_2,\ldots,x_m\in X$ with $m\leq C$ satisfying $X=\bigcup_{i=1}^mB_{\rho_k}(x_i,\epsilon)$. In this paper, we study the cases where $\{\rho_k\}$ is equal to $\{d_k^{\vec{v},b}\}$, $\{\hat{d}_k^{\vec{v},b}\}$ or $\{\bar{d}_k^{\vec{v},b}\}$.

  We also investigate the measure-theoretic complexity of invariant Borel probability measures. That is, for a given $\epsilon>0$ and $\mu\in M(X,T)$, we consider the measure complexity with respect to $\{\rho_k\}=\{d_k^{\vec{v},b}\}$, $\{\hat{d}_k^{\vec{v},b}\}$ or $\{\bar{d}_k^{\vec{v},b}\}$, defined by
  $$S(\rho_k)=\min\left\{m\in \mathbb{Z}_+:\exists x_1, \ldots,x_m\in X, \mu\left(\bigcup_{i=1}^mB_{\rho_k}(x_i,\epsilon)\right)>1-\epsilon\right\}.$$
We say $\mu\in M(X,T)$ has bounded complexity with respect to $\{\rho_k\}$ if for any $\epsilon>0$, there exists a positive integer $C=C(\epsilon)$ such that $S(\rho_k)\leq C$ for all $k\in\mathbb{N}$.
  Following the idea \cite{H1}, we introduce various notions of direcitonal equicontinuity, analogously to $\mathbb{Z}$-actions in Section 3 and Section 4. As expected, the bounded complexity of a $\mathbb{Z}^2$-system along a direction is related to various notions of directional equicontinuity.

Let us now formulate our main results in topological sense (recalled and proved below as Theorem \ref{thm1} and Theorem \ref{thm2}).
\begin{thm}
Let $(X,T)$ be a $\mathbb{Z}^2$-t.d.s., $\mu\in M(X,T)$ and $\vec{v}=(1,\beta)\in\mathbb{R}^2$ be a direction vector.  Then for any $b\in (0,\infty)$ the following two statements are equivalent.
\begin{itemize}
	\item[(a)]$(X,T)$ has bounded complexity with respect to $\{d_k^{\vec{v},b}\}$ $($resp. $\{\hat{d}_k^{\vec{v},b}\}$$)$.
	\item[(b)]$(X,T)$ is $(\vec{v},b)$-equicontinuous $($resp. $(\vec{v},b)$-equicontinuous in the mean$)$.
\end{itemize}
\end{thm}

 We also formulate our main results in measure-theoretic sense (recalled and proved below as Theorem \ref{thm3}, Theorem \ref{thm4},  Theorem \ref{thm7}, Theorem \ref{thm8} and Corollary \ref{cor1}). 
 \begin{thm}
 	Let $(X,T)$ be a $\mathbb{Z}^2$-t.d.s., $\mu\in M(X,T)$ and $\vec{v}=(1,\beta)\in\mathbb{R}^2$ be a direction vector. Then for any $b\in (0,\infty)$ the following two statements are equivalent.
 	\begin{itemize}
 		\item[(a)]$\mu$ has bounded complexity with respect to $\{d_k^{\vec{v},b}\}$.
 		\item[(b)]$T$ is $(\mu,\vec{v},b)$-equicontinuous.
 	\end{itemize}
\end{thm}
 
 \begin{thm}
 	Let $(X,T)$ be a $\mathbb{Z}^2$-t.d.s., $\mu\in M(X,T)$ and $\vec{v}=(1,\beta)\in\mathbb{R}^2$ be a direction vector. Then the following statements are equivalent.
 	\begin{itemize}
 		\item[(a)]$\mu$ has bounded complexity with respect to $\{\hat{d}_k^{\vec{v},b}\}$ for some $b\in (0,\infty)$ $($or for any $b\in (0,\infty))$.
 		\item[(b)]$\mu$ has bounded complexity with respect to $\{\bar{d}_k^{\vec{v},b}\}$ for some $b\in (0,\infty)$ $($or for any $b\in (0,\infty))$.
 		\item[(c)]$\mu$ has $\vec{v}$-discrete spectrum.
 		\item[(d)]$T$ is $(\mu,\vec{v},b)$-equicontinuous in the mean for some $b\in (0,\infty)$ $($or for any $b\in (0,\infty))$.
 		\item[(e)]$T$ is $(\mu,\vec{v},b)$-mean equicontinuous for some $b\in (0,\infty)$ $($or for any $b\in (0,\infty))$.
 	\end{itemize}
 	
 \end{thm}

  The structure of the paper is as follows. In Section 2, we recall some basic notions that we use in this paper. In Section 3, we show topological results for systems with bounded complexity, with respect to three kinds of metrics defined along any direction. In Section 4, we consider  corresponding results in the measure-theoretical setting. In Appendix A, we introduce the corresponding results of $\mathbb{Z}^q$-t.d.s. for any integer $q\geq 2$.
\section{Preliminaries}
In this section we recall some notions of dynamical systems that are used later (see \cite{EW,Ku1,Ku2,Peter}).
\subsection{General notions.} In this article, the sets of integers, non-negative integers, rational numbers and natural numbers are denoted by $\mathbb{Z}$, $\mathbb{Z}_+$, $\mathbb{Q}$ and $\mathbb{N}$, respectively. We use $\#(A)$ to denote the number of elements of a finite set $A$.

\subsection{Hausdorff metric} Let $K(X)$ be the hyperspace on $X$, i.e. the space of non-empty and closed subsets of $X$ equipped with the Hausdorff metric $d_H$ defined by
$$d_H(A,B)=\max\left\{\max_{x\in A}\min_{y\in B}d(x,y),\max_{y\in B}\min_{x\in A}d(x,y)\right\}$$
for $A,B\in K(X)$. As $(X,d)$ is compact, $(K(X),d_H)$ is compact as well. For $n\in \mathbb{N}$, it is easy to see that the map $X^n\to K(X),(x_1,\ldots,x_n)\mapsto\{x_1,\ldots,x_n\}$ is continuous. Then $\{A\in K(X):\#(A)\leq n\}$ is a closed subset of $K(X)$.

\subsection{Bowen metric, mean metric and max-mean metric for $\mathbb{Z}$-actions} Let $(X,T)$ be a $\mathbb{Z}$-t.d.s. with a metric $d$.
For $k\in \mathbb{N}$, we define three metrics on $X$ as follows. For $x,y \in X$, let
$$d_k(x,y)=\max\{d(T^ix,T^iy):i\in[0,k-1]\},$$
 $$\bar{d}_k(x,y)=\frac{1}{k}\sum_{i=0}^{k-1}d(T^ix,T^iy)$$
and $$\hat{d}_k(x,y)=\max\left\{\frac{1}{j}\sum_{i=0}^{j-1}d(T^ix,T^iy):j\in [1,k]\right\},$$
which is called Bowen metric, mean metric and max-mean metric, respectively.
It is clear that for all $k\in \mathbb{N}$ and $x,y\in X$,
$$d_k(x,y)\geq \hat{d}_k(x,y)\geq  \bar{d}_k(x,y).$$

\subsection{Notions of equicontinuity} In order to study the chaotic behaviors of dynamical systems, Huang, Lu and Ye \cite{HLY} introduced a notion that reflects the equicontinuity with respect to a subset or a measure. Following \cite{HLY}, for a $\mathbb{Z}$-t.d.s. $(X,T)$, we say a subset $K$ of $X$ is equicontinuous if for any $\epsilon>0$, there exists $\delta>0$ such that $d(T^nx,T^ny)<\epsilon$ for all $n\in \mathbb{Z}_+$ and all $x,y\in K$ with $d(x,y)<\delta$. Given $\mu\in M(X,T)$, we say that $T$ is $\mu$-equicontinuous if for any $\tau>0$, there exists a $T$-equicontinuous measurable subset $K$ of $X$ with $\mu(K)>1-\tau$.

When studying the relation between bounded complexity, mean equicontinuity and discrete spectrum for $\mathbb{Z}$-actions, Huang et al. \cite{H1} introduced notions called  $\mu$-mean equicontinuity and $\mu$-equicontinuity in the mean. Following \cite{H1}, for a  $\mathbb{Z}$-t.d.s. $(X,T)$, we say a subset $K$ of $X$ is equicontinuous in the mean if for any $\epsilon>0$, there exists $\delta>0$ such that $\bar{d}_n(x,y)<\epsilon$ for all $n\in \mathbb{N}$ and $x,y\in K$ with $d(x,y)<\delta$. We say a subset $K$ of $X$ is mean equicontinuous if for any $\epsilon>0$, there exists $\delta>0$ such that $\limsup\limits_{n\to \infty}\bar{d}_n(x,y)<\epsilon$ for all $x,y\in K$ with $d(x,y)<\delta$. Given $\mu\in M(X,T)$, we say that $T$ is $\mu$-equicontinuous in the mean if for any $\tau>0$, there exists a measurable subset $K$ of $X$ with $\mu(K)>1-\tau$ that is equicontinuous in the mean and $\mu$-mean equicontinuous if for any $\tau>0$, there exists a measurable subset $K$ of $X$ with $\mu(K)>1-\tau$ that is mean equicontinuous.

\subsection{Kronecker algebra and discrete spectrum.}
\subsubsection{Kronecker algebra and discrete spectrum for $\mathbb{Z}$-actions.}
In this subsection, let $(X,T)$ be a $\mathbb{Z}$-t.d.s., $\mu\in M(X,T)$ and $\mathcal{H}:=L^2(X,\mathcal{B}_X,\mu)$. In the complex Hilbert space $\mathcal{H}$, we define the unitary operator $U_T:\mathcal{H}\rightarrow \mathcal{H}$ by $U_Tf=f\circ T$ for any $f\in\mathcal{H}$.  We say that $f$ is an almost periodic function if $\overline{\left\{U_T^n f:n\in \mathbb{Z}\right\}}$ is a compact subset of $\mathcal{H}$.
As one can easily prove that the set of all bounded almost periodic functions forms a $U_T$-invariant and conjugation-invariant subalgebra of $\mathcal{H}$ (denoted by $\mathcal{A}_c$). The set of almost periodic functions is just the closure of $\mathcal{A}_c$ (denoted by $\mathcal{H}_c$). It is well known that there exists a $T$-invariant sub-$\sigma$-algebra $\mathcal{K}_{\mu}$ of $\mathcal{B}_X$ such that $\mathcal{H}_c=L^2(X,\mathcal{K}_{\mu},\mu)$ (see \cite[Theorem 1.2]{Zi}). The sub-$\sigma$-algebra $\mathcal{K}_{\mu}$ is called the Kronecker algebra of $(X,\mathcal{B}_X,\mu,T)$.    From the constrution of $\mathcal{H}_c$, we can see  $\mathcal{K}_\mu$ consists of all $B\in \mathcal{B}_X$ such that
$\overline{\{ U_T^n1_B:n\in \mathbb{Z} \}}$ is compact in $\mathcal{H}.$
We say $\mu\in M(X,T)$ has discrete spectrum if $\mathcal{B}_X=\mathcal{K}_{\mu}$.
\subsubsection{Directional Kronecker algebra and discrete spectrum \cite{LX}.}
  In this subsection, let $(X ,T)$ be a $\mathbb{Z}^2$-t.d.s. Let $\mu\in M(X,T)$, $\vec{v}=(1,\beta)\in\mathbb{R}^2$ be a direction vector and $b\in (0,\infty)$. We put $$\Lambda^{\vec{v}}(b)=\left\{(m,n)\in\mathbb{Z}^2:\beta m-b\leq n\leq \beta m+b\right\}.$$ Let $\mathcal{A}_c^{\vec{v}}(b)$ be the collection of $f\in \mathcal{H}:=L^2(X,\mathcal{B}_X,\mu)$ such that
$$\overline{\left\{U_T^{(m,n)}f:(m,n)\in \Lambda^{\vec{v}}(b) \right\}}\text{ is compact in }\mathcal{H}.$$
It is easy to see that  $\mathcal{A}_c^{\vec{v}}(b)$ is a $U_{T^{\vec{w}}}$-invariant for all $\vec{w}\in\mathbb{Z}^2$ and conjugation-invariant subalgebra of $\mathcal{H}$. It is well known that (see \cite[Theorem 1.2]{Zi}) there exists a $T$-invariant sub-$\sigma$-algebra  $\mathcal{K}_\mu^{\vec{v}}(b)$  of $\mathcal{B}_X$ such that
\begin{align}\label{1}\mathcal{A}_c^{\vec{v}}(b)=L^2(X,\mathcal{K}_\mu^{\vec{v}}(b),\mu).\end{align}
Directly from \eqref{1}, the $\vec{v}$-directional Kronecker algebra of $(X ,\mathcal{B}_X, \mu, T)$ can be defined by
$$\mathcal{K}_\mu^{\vec{v}}(b)=\left\{B\in\mathcal{B}_X: \overline{\left\{U_T^{(m,n)}1_B :(m,n)\in \Lambda^{\vec{v}}(b) \right\}}\text{ is compact in } L^2(X,\mathcal{B}_X,\mu) \right\}.$$
Since the definition of $\mathcal{K}_\mu^{\vec{v}}(b)$ is independent of the selection of $b\in (0,\infty)$ (refer to \cite{LX}), we omit $b$ and write $\mathcal{K}_\mu^{\vec{v}}(b)$ as $\mathcal{K}_\mu^{\vec{v}}$. We say $\mu\in M(X,T)$ has $\vec{v}$-discrete spectrum if $\mathcal{K}_\mu^{\vec{v}}=\mathcal{B}_X$.

   \section{Topological dynamical systems with directional bounded topological complexity}
   In this section, we study the topological complexity of dynamical systems with respect to three kinds of metrics defined along directions.
   \subsection{Topological complexity with respect to $\{d_k^{\vec{v},b}\}$ }
  Let $(X,T)$ be a  $\mathbb{Z}^2$-t.d.s. with a metric $d$. Let   $\vec{v}=(1,\beta)\in\mathbb{R}^2$ be a direction vector and $b\in (0,\infty)$.
   We put $$\Lambda^{\vec{v}}(b)=\left\{(m,n)\in\mathbb{Z}^2:\beta m-b\leq n\leq \beta m+b\right\}.$$
   For $k\in \mathbb{N}$ and $x,y\in X$, define
   $$d_k^{\vec{v},b}(x,y)=\max_{(m,n)\in \Lambda_k^{\vec{v}}(b) }\{d(T^{(m,n)}x,T^{(m,n)}y)\}.$$
   It is easy to see that for each $k\in \mathbb{N}$, $d_k^{\vec{v},b}$ is a metric on $X$ which is topologically equivalent to the metric $d$. Given $x\in X$ and $\epsilon>0$, the open ball of center $x$ and radius $\epsilon$ in the metric $d_k^{\vec{v},b}$ is
   $$B_{d_k^{\vec{v},b}}(x,\epsilon)=\{y\in X: d_k^{\vec{v},b}(x,y)<\epsilon\}.$$
   Let $K$ be a subset of $X$, $k\in \mathbb{N}$ and $\epsilon>0$. A subset $F$ of $K$ is said to $(k,\epsilon)$-span $K$ with respect to $T$, if  for every $x\in K$,  there exists $y\in F$ with $d_k^{\vec{v},b}(x,y)<\epsilon$, that is
   $$K\subset \bigcup_{x\in F}B_{d_k^{\vec{v},b}}(x,\epsilon).$$
   Let $span_K^{\vec{v},b}(k,\epsilon)$ denote the smallest cardinality of any $(k,\epsilon)$-spanning set for $K$ with respect to $T$, that is  $$span^{\vec{v},b}_K(k,\epsilon)=\min\left\{\#(F):F\subset X\subset \bigcup_{x\in F}B_{d_k^{\vec{v},b}}(x,\epsilon) \right\}.$$
   We say that a subset $K$ of $X$ has bounded topological complexity with respect to $\{d_k^{\vec{v},b}\}$ if for any $\epsilon>0$, there exists a positive integer $C=C(\epsilon)$ such that $span_K^{\vec{v},b}(k,\epsilon)\leq C$ for all $k\in \mathbb{N}$.
   
    Meanwhile, we say that a subset $K$ of $X$ is $(\vec{v},b)$-equicontinuous if for any $\epsilon>0$, there is $\delta>0$ such that whenever $x,y\in K$ with $d(x,y)<\delta$, $d(T^{(m,n)}x,T^{(m,n)}y)<\epsilon$ for all $(m,n)\in \Lambda^{\vec{v}}(b)$.

   We first show that a subset with bounded topological complexity with respect to $\{d_k^{\vec{v},b}\}$ is equivalent to the $(\vec{v},b)$-equicontinuity property.
   \begin{thm}\label{thm1}
 Let $(X,T)$ be a  $\mathbb{Z}^2$-t.d.s., $K$ be a compact subset of $X$ and   $\vec{v}=(1,\beta)\in\mathbb{R}^2$ be a direction vector. Then for any $b\in(0,\infty)$ the following two statements are equivalent.
 \begin{itemize}
 	\item[(a)]$K$ has bounded topological complexity with respect to $\{d_k^{\vec{v},b}\}$.
 	\item[(b)]$K$ is $(\vec{v},b)$-equicontinuous.
 \end{itemize}

\end{thm}
\begin{proof}
(b) $\Rightarrow$ (a). Fix $\epsilon>0$. By the definition of $(\vec{v},b)$-equicontinuity, there exists $\delta>0$ such that $d(T^{(m,n)}x,T^{(m,n)}y)<\epsilon$ for all $(m,n)\in \Lambda^{\vec{v}}(b)$ and $x,y\in K$ with $d(x,y)<\delta$. By the compactness of $K$, there exists a finite subset $F$ of $K$ such that $K\subset  \bigcup_{x\in F}B_{d}(x,\delta)$. Then $K\subset \bigcup_{x\in F}B_{d_k^{\vec{v},b}}(x,\epsilon)$ for all $k\in \mathbb{N}$. So $K$ has bounded topological complexity with respect to $\{d_k^{\vec{v},b}\}$.

(a) $\Rightarrow$ (b). Suppose the contrary that $K$ is not $(\vec{v},b)$-equicontinuous. There exists  $\epsilon>0$, for any $l\in \mathbb{N}$, there are $x_l,y_l\in K$ and $(m_l,n_l)\in \Lambda^{\vec{v}}(b)$ such that $d(x_l,y_l)<\frac{1}{l}$ but $d(T^{(m_l,n_l)}x_l,T^{(m_l,n_l)}y_l)\geq\epsilon$. Passing to a subsequence if necessary, we may assume that $\lim\limits_{l\to\infty}x_l=x_0$. Then we have $x_0\in K$ and $\lim\limits_{l\to\infty}y_l=x_0$. For any $l\in\mathbb{N}$, by the triangle inequality, either $d(T^{(m_l,n_l)}x_l,T^{(m_l,n_l)}x_0)\geq\epsilon/2$ or $d(T^{(m_l,n_l)}y_l,T^{(m_l,n_l)}x_0)\geq\epsilon/2$.
Without loss of generality, we may assume $d(T^{(m_l,n_l)}x_l,T^{(m_l,n_l)}x_0)\geq\epsilon/2$ for all $l\in \mathbb{N}$.

As $K$ has bounded topological complexity with respect to $\{d_k^{\vec{v},b}\}$, for the constant $\epsilon/6$, there exists $C=C(\epsilon/6)>0$ such that for every $k\in \mathbb{N}$, there exists a subset $F_k$ of $K$ with $\#(F_k)\leq C$ such that $K\subset\bigcup_{x\in F_k}B_{d_k^{\vec{v},b}}(x,\epsilon/6)$. We view $\{F_k\}$ as a sequence in the hyperspace $K(X)$. By the compactness of $K(X)$, there is a subsequence $\{F_{k_l}\}$ of $\{F_k\}$ such that $\lim\limits_{l\to \infty}F_{k_l}=F$ in the Hausdorff metric $d_H$. Since $F_k\subset K$ for all $k\in \mathbb{N}$ and $K$ is compact, we have $F\subset K$. By the fact $\{A\in K(X): \#(A)\leq C\}$ is closed, we have $\#(F)\leq C$. For any $l\in \mathbb{N}$ and $x\in K$, there exists $z_{k_l}\in F_{k_l}$ such that $d_{k_l}^{\vec{v},b}(z_{k_l},x)<\epsilon/6$. Without loss of generality, we may assume that $\lim\limits_{l\to \infty}z_{k_l}=z$. Then $z\in F$. As the sequence $\{d_k^{\vec{v},b}\}$ of metirc is increasing, that is, $d_k^{\vec{v},b}(u,v)\leq d_{k+1}^{\vec{v},b}(u,v)$ for all $u,v\in X$ and $k\in \mathbb{N}$, we have $$d_{k_l}^{\vec{v}}(x,z_{k_t})\leq d_{k_t}^{\vec{v}}(x,z_{k_t})<\epsilon/6$$ for all $t\geq l$. Let $t\to \infty$. Then we get $d_{k_l}^{\vec{v},b}(x,z)\leq \epsilon/6$. This implies that $$K= \bigcup_{z\in F}\{x\in K : d_{k_l}^{\vec{v},b}(x,z)\leq \epsilon/6\}$$ for all $l\in\mathbb{N}$. Enumerate $F$ as $\{z_1,\ldots z_r\}$ and let
$$K_t=\bigcap_{k=1}^{\infty}\{x\in K:d_{k}^{\vec{v},b}(x,z_t)\leq \epsilon/6\}$$
for $t\in\{1,2,\ldots,r\}$. Then each $K_t$ is a closed set, since $d_k^{\vec{v},b}$ is continuous. By the monotonicity of $\{d_k^{\vec{v},b}\}$, we have $K=\bigcup_{t=1}^r{K_t}$.

For the sequence $\{x_l\}$ in $K$, passing to a subsequence if necessary, we assume that the sequence $\{x_l\}$ in the same $K_t$ for some $t\in\{1,2,\ldots,r\}$. As $K_t$ is closed, $x_0$ is also in $K_t$. Note that for any $u,v\in K_t$ and $k\in \mathbb{N}$, we have $$d_k^{\vec{v},b}(u,v)\leq d_k^{\vec{v},b}(u,z_t)+d_k^{\vec{v},b}(z_t,v)\leq \epsilon/3.$$
Particularly, we have $d_{k}^{\vec{v},b}(x_l,x_0)\leq \epsilon/3$ for any $l,k\in \mathbb{N}$, which contradicts the above assumption that is $d(T^{(m_l,n_l)}x_l,T^{(m_l,n_l)}x_0)\geq\epsilon/2$ for all $l\in \mathbb{N}$. Now we finish the proof of Theorem \ref{thm1}.
\end{proof}

\subsection{Topological complexity with respect to $\{\hat{d}_k^{\vec{v},b}\}$}  Let $(X,T)$ be a  $\mathbb{Z}^2$-t.d.s. with a metric $d$. Let   $\vec{v}=(1,\beta)\in\mathbb{R}^2$ be a direction vector and $b\in (0,\infty)$.
 Define
$$\hat{d}_k^{\vec{v},b}(x,y)=\max_{1\leq i\leq k}\left\{\frac{1}{\#(\Lambda_i^{\vec{v}}(b))}\sum_{(m,n)\in\Lambda_i^{\vec{v}}(b)}d(T^{(m,n)}x,T^{(m,n)}y)\right\}$$
for $k\in \mathbb{N}$ and $x,y\in X$.
It is easy to see that for each $k\in \mathbb{N}$, $\hat{d}_k^{\vec{v},b}$ is a metric on $X$ which is topologically equivalent to the metric $d$. For $x\in X$ and $\epsilon>0$, let $B_{\hat{d}_k^{\vec{v},b}}(x,\epsilon)=\{y\in X: \hat{d}_k^{\vec{v},b}(x,y)<\epsilon\}.$ Let $K$ be a subset of $X$. For $k\in \mathbb{N}$ and $\epsilon>0$, define
$$\widehat{span}_K^{\vec{v},b}(k,\epsilon)=\min\left\{\#(F):F\subset X\subset \bigcup_{x\in F}B_{\hat{d}_k^{\vec{v},b}}(x,\epsilon)\right\}.$$
We say that a subset $K$ of $X$ has bounded topological complexity with respect to $\{\hat{d}_k^{\vec{v},b}\}$ if for any $\epsilon>0$, there exists a positive integer $C=C(\epsilon)$ such that $\widehat{span}_K^{\vec{v},b}(k,\epsilon)\leq C$ for all $k\in \mathbb{N}$. Since $\hat{d}_k^{\vec{v},b}(x,y)\leq d_k^{\vec{v},b}(x,y)$ for all $k\in \mathbb{N}$ and $x,y\in X$, if $K$ has bounded topological complexity with respect to $\{d_k^{\vec{v},b}\}$, then it has bounded topological complexity with respect to $\{\hat{d}_k^{\vec{v},b}\}$ as well.

 We say that a subset $K$ of $X$ is $(\vec{v},b)$-equicontinuous in the mean if for any $\epsilon>0$, there is $\delta>0$ such that $\hat{d}_k^{\vec{v},b}(x,y)<\epsilon$ for all $k\in\mathbb{N}$ and $x,y\in X$ with $d(x,y)<\delta$.
Corresponding with mean equicontinuity, we say a subset $K$ of $X$ is $(\vec{v},b)$-mean equicontinuous if for any $\epsilon>0$, there exists $\delta>0$ such that
$$\limsup_{k\to \infty}\frac{1}{\#(\Lambda_k^{\vec{v}}(b))}\sum_{(m,n)\in\Lambda_k^{\vec{v}}(b)}d(T^{(m,n)}x,T^{(m,n)}y)<\epsilon$$
for all $x,y\in K$ with $d(x,y)<\delta$. If $X$ is $(\vec{v},b)$-mean equicontinuous, then we say that $(X,T)$ is $(\vec{v},b)$-mean equicontinuous. By the definition, it is clear that if $K$ is $(\vec{v},b)$-equicontinuous in the mean, then it is $(\vec{v},b)$-mean equicontinuous.

 The following result follows the same idea lies in Theorem \ref{thm1} and just replaces the distance $d_k^{\vec{v},b}$ by $\hat{d}_k^{\vec{v},b}$, as the sequence $\{\hat{d}_k^{\vec{v},b}\}$ of metrics is also increasing.
 \begin{thm}\label{thm2}
Let $(X,T)$ be a  $\mathbb{Z}^2$-t.d.s., $K$ be a compact subset of $X$ and  $\vec{v}=(1,\beta)\in\mathbb{R}^2$ be a direction vector.  Then for any $b\in (0,\infty)$ the following two statements are equivalent.
\begin{itemize}
	\item[(a)]$K$ has bounded topological complexity with respect to $\{\hat{d}_k^{\vec{v},b}\}$.
	\item[(b)]$K$ is $(\vec{v},b)$-equicontinuous in the mean.
\end{itemize}
 \end{thm}

\subsection{Topological complexity with respect to $\{\bar{d}_k^{\vec{v},b}\}$ }
Let $(X,T)$ be a  $\mathbb{Z}^2$-t.d.s. with a metric $d$. Let   $\vec{v}=(1,\beta)\in\mathbb{R}^2$ be a direction vector and $b\in (0,\infty)$.
For $k\in \mathbb{N}$ and $x,y\in X$, define
$$\bar{d}_k^{\vec{v},b}(x,y)=\frac{1}{\#(\Lambda_k^{\vec{v}}(b))}\sum_{(m,n)\in\Lambda_k^{\vec{v}}(b)}d(T^{(m,n)}x,T^{(m,n)}y).$$
It is easy to see that for each $k\in \mathbb{N}$, $\bar{d}_k^{\vec{v},b}$ is a metric on $X$ which is topologically equivalent to the metric $d$. For $x\in X$ and $\epsilon>0$, let $B_{\bar{d}_k^{\vec{v},b}}(x,\epsilon)=\{y\in X: \bar{d}_k^{\vec{v},b}(x,y)<\epsilon\}.$ Let $K$ be a subset of $X$. For $k\in \mathbb{N}$ and $\epsilon>0$, define
$$\overline{span}^{\vec{v},b}_K(k,\epsilon)=\min\left\{\#(F):F\subset K\subset \bigcup_{x\in F}B_{\bar{d}_k^{\vec{v},b}}(x,\epsilon) \right\}.$$
We say that a subset $K$ of $X$ has bounded topological complexity with respect to $\{\bar{d}_k^{\vec{v},b}\}$ if for any $\epsilon>0$, there exists a positive integer $C=C(\epsilon)$ such that $\overline{span}_K^{\vec{v},b}(k,\epsilon)\leq C$ for all $k\in \mathbb{N}$. 

Since $\bar{d}_k^{\vec{v},b}(x,y)\leq \hat{d}_k^{\vec{v},b}(x,y)$ for all $k\in \mathbb{N}$ and $x,y\in X$, if $K$ has bounded topological complexity with respect to $\{\hat{d}_k^{\vec{v},b}\}$, then it has bounded topological complexity with respect to $\{\bar{d}_k^{\vec{v},b}\}$ as well. 
However, as the sequence $\{\bar{d}_k^{\vec{v},b}\}$ of metrics might be not monotonous, dynamical systems with bounded topological complexity with respect to $\{\bar{d}_k^{\vec{v},b}\}$ does not have the similar properties to those with respect to $\{\hat{d}_k^{\vec{v},b}\}$ or $\{d_k^{\vec{v},b}\}$.  Huang et al. \cite[Proposition 3.8 and Propostion 3.9]{H1} constructed examples to show that there exists a  $\mathbb{Z}$-t.d.s. that has bounded complexity with respect to $\{\bar{d}_k^{\vec{v},b}\}$ but it is not $(\vec{v},b)$-equicontinuous, where $\vec{v}=(1,0)$ and $b=1/10$. Thus the relation between bounded complexity with respect to $\{\bar{d}_k^{\vec{v},b}\}$ and $(\vec{v},b)$-equicontinuity is more confusion. However, if a dynamical system has bounded topological complexity with respect to $\{\bar{d}_k^{\vec{v},b}\}$ for some $b\in(0,\infty)$, then by Theorem \ref{thm6} in Section 4 every invariant Borel probability measure has $\vec{v}$-discrete spectrum. So it is simple in the measure-theoretic sense even for irrational directions.

\section{Invariant measures with bounded-theoretic complexity }
In this section, we investigate the measure-theoretic complexity of invariant Borel probability measures with respect to three kinds of metrics defined along directions.
\subsection{Measure-theoretic complexity with respect to $\{d_k^{\vec{v},b}\}$}
 Let $(X,T)$ be a  $\mathbb{Z}^2$-t.d.s. with a metric $d$ and $\mu\in M(X,T)$. Let   $\vec{v}=(1,\beta)\in\mathbb{R}^2$ be a direction vector and $b\in (0,\infty)$. For $k\in \mathbb{N}$ and $\epsilon>0$, let
 $$span_{\mu}^{\vec{v},b}(k,\epsilon)=\min\left\{\#(F):F\subset X\text{ and } \mu\left(\bigcup_{x\in F}B_{d_k^{\vec{v},b}}(x,\epsilon)\right)>1-\epsilon\right\}.$$
We say that $\mu$ has bounded complexity with respect to $\{d_k^{\vec{v},b}\}$ if for any $\epsilon>0$, there exists a positive integer $C=C(\epsilon)$ such that $span_{\mu}^{\vec{v},b}(k,\epsilon)\leq C$ for all $k\in \mathbb{N}$. We say that $T$ is $(\mu,\vec{v},b)$-equicontinuous if for any $\tau>0$, there exists a $(\vec{v},b)$-equicontinuous measurable subset $K$ of $X$ with $\mu(K)>1-\tau$.

The following result shows that $\mu\in M(X,T)$ with bounded complexity with respect to $\{d_k^{\vec{v},b}\}$ if and only if $T$ is  $(\mu,\vec{v},b)$-equicontinuious.
\begin{thm}\label{thm3}
 Let $(X,T)$ be a  $\mathbb{Z}^2$-t.d.s., $\mu\in M(X,T)$ and  $\vec{v}=(1,\beta)\in\mathbb{R}^2$ be a direction vector. Then for any $b\in (0,\infty)$ the following two statements are equivalent.
 \begin{itemize}
\item[(a)]$\mu$ has bounded complexity with respect to $\{d_k^{\vec{v},b}\}$.
\item[(b)]$T$ is $(\mu,\vec{v},b)$-equicontinuous.
 \end{itemize}
 \begin{proof}
(b) $\Rightarrow$ (a). We assume that $T$ is $(\mu,\vec{v},b)$-equicontinuous. For any $\epsilon>0$, there exists a $(\vec{v},b)$-equicontinuous measurable subset $K$ of $X$ with $\mu(K)>1-\epsilon$. As the measure $\mu$ is regular, we can require the set $K$ to be compact. Now we obtain the result from Theorem \ref{thm1}, since $span_{\mu}^{\vec{v},b}(k,\epsilon)\leq span_K^{\vec{v},b}(k,\epsilon).$

(a) $\Rightarrow$ (b). For any $\tau>0$, we need to prove that there exists a  measurable subset $(\vec{v},b)$-equicontinuous set $K$ of $X$ with $\mu(K)>1-\tau$. Fix $\tau>0$. As $\mu$ has bounded complexity with respect to $\{d_k^{\vec{v},b}\}$, for any $M>0$, there exists  $C_M>0$ such that for every $k\in \mathbb{N}$, there exists a subset of $F_k$ of $X$ with $\#(F_k)\leq C_M$ such that
$$\mu\left(\bigcup_{x\in F_k}B_{d_k^{\vec{v},b}}\left(x,1/M\right)\right)>1-\frac{\tau}{2^{M+2}}.$$
Since the measure $\mu$ is regular, we can pick a compact subset $K_k$ of $\bigcup_{x\in F_k}B_{d_k^{\vec{v},b}}(x,1/M)$ with $\mu(K_k)>1-\frac{\tau}{2^{M+2}}$ for each $k\in\mathbb{N}$. Without loss of generality, we can assume that $\lim\limits_{k\to \infty}F_k=F_M$ and $\lim\limits_{k\to \infty}K_k=K_M$ in the Hausdorff metric. Then $\#(F_M)\leq C_M$. As $K_k$ is closed,
$$\mu(K_M)\geq\limsup_{k\to \infty}\mu(K_k)\geq 1-\frac{\tau}{2^{M+2}}.$$
For any $x\in K_M$ and $k\in \mathbb{N}$, there exists  $N>0$ such that for any $l>N$, there exist $x_l\in K_l$ and $y_l\in F_l$ such that $d_k^{\vec{v},b}(x,x_l)<1/M$ and $d_l^{\vec{v},b}(x_l,y_l)<1/M$. Passing to a subsequence if necessary, we may assume that $\lim\limits_{l\to \infty}y_l=y$. Then $y\in F_M$. By the monotonicity of $\{d_k^{\vec{v},b}\}$, for any $l>k$, we have
\begin{align*}
d_k^{\vec{v},b}(x,y_l)\leq& d_k^{\vec{v},b}(x,x_l)+d_k^{\vec{v},b}(x_l,y_l)\\
\leq& d_k^{\vec{v},b}(x,x_l)+d_l^{\vec{v},b}(x_l,y_l)
\leq 2/M.
\end{align*}
Letting  $l\to \infty$, we have $d_k^{\vec{v},b}(x,y)\leq2/M$. Then we obtain that for each $M\in\mathbb{N}$
$$K_M\subset \bigcup_{x\in F_M}B_{d_k^{\vec{v},b}}\left(x,3/M\right),$$
and hence
$$span_{K_M}^{\vec{v},b}\left(k,3/M\right)\leq \#(F_M)\leq C_M.$$
Let $K=\bigcap_{M=1}^{\infty}K_M$. Then $\mu(K)>1-\tau$ and for any $M\in \mathbb{N}$,
\begin{equation}\label{key}
span_{K}^{\vec{v},b}\left(k,3/M\right)\leq span_{K_M}^{\vec{v},b}\left(k,3/M\right)\leq C_M
\end{equation}
for all $k\in \mathbb{N}$. Now, by \eqref{key} and Theorem \ref{thm1}, we know that $K$ is $(\vec{v},b)$-equicontinuous. As $\tau$ is arbitrart, $T$ is $(\mu,\vec{v},b)$-equicontinuous.
We finish the proof of Theorem \ref{thm3}.
 \end{proof}
\end{thm}

 \subsection{Measure-theoretic complexity with respect to $\{\hat{d}_k^{\vec{v},b}\}$} Let $(X,T)$ be a  $\mathbb{Z}^2$-t.d.s. with a metric $d$ and $\mu\in M(X,T)$. Let   $\vec{v}=(1,\beta)\in\mathbb{R}^2$ be a direction vector and $b\in (0,\infty)$. For $k\in \mathbb{N}$ and $\epsilon>0$,  let
 $$\widehat{span}_{\mu}^{\vec{v},b}(k,\epsilon)=\min\left\{\#(F):F\subset X\text{ and } \mu\left(\bigcup_{x\in F}B_{\hat{d}_k^{\vec{v},b}}(x,\epsilon)\right)>1-\epsilon\right\}.$$
 We say that $\mu$ has bounded complexity with respect to $\{\hat{d}_k^{\vec{v},b}\}$ if for any $\epsilon>0$, there exists a positive integer $C=C(\epsilon)$ such that $\widehat{span}_{\mu}^{\vec{v},b}(k,\epsilon)\leq C$ for all $k\in \mathbb{N}$.
 
We say that $T$ is $(\mu,\vec{v},b)$-equicontinuous in the mean if for any $\tau>0$, there exists a measurable subset $K$ of $X$ with $\mu(K)>1-\tau$ that is $(\vec{v},b)$-equicontinuous in the mean, and $(\mu,\vec{v},b)$-mean equicontinuous if for any $\tau>0$, there exists a measurable subset $K$ of $X$ with $\mu(K)>1-\tau$ that is $(\vec{v},b)$-mean equicontinuous.

The following result implies that the relation between $\mu\in M(X,T)$ with bounded complexity with respect to $\{\hat{d}_k^{\vec{v},b}\}$ and two kinds of  $(\mu,\vec{v},b)$-equicontinuity metioned above. 
 \begin{thm}\label{thm4}
Let $(X,T)$ be a  $\mathbb{Z}^2$-t.d.s., $\mu\in M(X,T)$ and $\vec{v}=(1,\beta)\in\mathbb{R}^2$ be a direction vector. Then for any $b\in (0,\infty)$ the following statements are equivalent.
\begin{itemize}
\item[(a)]$\mu$ has bounded complexity with respect to $\{\hat{d}_k^{\vec{v},b}\}$.
\item[(b)]$T$ is $(\mu,\vec{v},b)$-equicontinuous in the mean.
\item[(c)]$T$ is $(\mu,\vec{v},b)$-mean equicontinuous.
\end{itemize}

\begin{proof}
(a) $\Rightarrow$ (b). Following the proof of Theorem \ref{thm3}, we know that for a given $\tau>0$, there exists a compact subset $K$ such that $\mu(K)\geq 1-\tau$ and for any $M\geq 1$, $\widehat{span}_K^{\vec{v},b}(k,6/M)\leq C_M$ for all $k\in \mathbb{N}$. It follows from Theorem \ref{thm2} that $K$ is $(\vec{v},b)$-equicontinuous in the mean. This proves that $T$ is $(\mu,\vec{v},b)$-equicontinuous in the mean.

(b) $\Rightarrow$ (c). It is directly from definitions.

(c) $\Rightarrow$ (a). Now assume that $T$ is $(\mu,\vec{v},b)$-mean equicontinuous. For any $\epsilon>0$, there exists a compact subset $K$ of $X$ such that $\mu(K)>1-2\epsilon$ and $K$ is $(\vec{v},b)$-mean equicontinuous. Namely, there exists $\delta>0$ such that  $$\limsup_{k\to \infty}\frac{1}{\#(\Lambda_k^{\vec{v}}(b) )}\sum_{(m,n)\in\Lambda_k^{\vec{v}}(b)}d(T^{(m,n)}x,T^{(m,n)}y)<\epsilon/4$$
for all $x,y\in K$ with $d(x,y)<\delta$. As $K$ is compact, there is a finite subset $F$ of $K$ such that $K\subset \bigcup_{x\in F}B_d(x,\delta)$. Enumerate $F$ as $\{x_1,\ldots,x_m\}$. For $j\in\{1,2,\ldots,m\}$ and $N\in\mathbb{N}$, let
\begin{equation*}
\begin{split}
A_N(x_j)=\{y\in B_d(x_j,\delta)\cap K:\frac{1}{\#(\Lambda_k^{\vec{v}}(b))}\sum_{(m,n)\in\Lambda_k^{\vec{v}}(b)}d(T^{(m,n)}x_j,T^{(m,n)}y)
<\epsilon/2, k\geq N \}.
\end{split}
\end{equation*}
It is easy to see that for each $j\in\{1,2,\ldots,m\}$,  $\{A_N(x_j)\}_{N=1}^{\infty}$ is an increasing sequence and $B_d(x_j,\delta)\cap K=\bigcup_{N=1}^{\infty}A_N(x_j)$. Choose $N_1\in \mathbb{N}$ and a compact subset $K_1$ of $A_{N_1}(x_1) $ such that $\mu(K_1)>\mu(B_d(x_1,\delta)\cap K)-\epsilon/2^m$.
Choose $N_2\in \mathbb{N}$ and a compact subset $K_2$ of $A_{N_2}(x_2) $ such that $K_1\cap K_2=\emptyset$ and $\mu(K_1 \cup K_2)>\mu((B_d(x_1,\delta)\cup B_d(x_2,\delta))\cap K)-\epsilon/2^{m-1}$. We can choose compact subsets $K_j$ of $A_{N_j}(x_j)$  for each $j\in\{1,2,\ldots,m\}$  inductively such that $\mu(\bigcup_{j=1}^m K_j)>\mu(K)-\epsilon/2>1-\epsilon$   and $K_i\cap K_j=\emptyset$ for any $1\leq i<j\leq m$.

Let $K_0=\bigcup_{j=1}^m K_j$ and $N_0=\max_{1\leq j\leq m}\{N_j\}$. Then we can choose  $\delta_1>0$ such that for every $x,y\in K_0$ with $d(x,y)<\delta_1$, there exists $j\in\{1,2,\ldots,m\}$ with $x,y\in K_j$. By the continuity of $T$, there exists $\delta_2>0$ such that $d_{N_0}^{\vec{v},b}(x,y)<\epsilon$ for all $x,y\in X$ with $d(x,y)<\delta_2$. Put $\delta_3=\min\{\delta_1,\delta_2\}$. By the compactness of $K_0$, there exists a finite subset $H$ of $K_0$ such that $K_0\subset\bigcup_{x\in H}B_d(x,\delta_3)$. Given any $y\in K_0$, we will show that there exists $x\in H$ such that $y\in B_{\hat{d}^{\vec{v},b}_k}(x,\epsilon)$ for all $k\in\mathbb{N}$. If $k<N_0$, then there exists $x\in H$ such that $y\in B_d(x,\delta_3)$ and hence $\hat{d}^{\vec{v},b}_k(x,y)\leq d_{N_0}^{\vec{v},b}(x,y)<\epsilon$. If $k\geq N_0$, then there exist $x\in H$ and $j\in\{1,2,\ldots,m\}$ such that $x,y \in K_j\subset A_{N_j}(x_j)$. According to the construction of $A_{N_j}(x_j)$ and $k\geq N_j$,
\begin{align*}
&\hat{d}^{\vec{v},b}_k(x,y)=\frac{1}{\#(\Lambda_k^{\vec{v}}(b))}\sum_{(m,n)\in\Lambda_k^{\vec{v}}(b)}d(T^{(m,n)}x,T^{(m,n)}y)\\
\leq& \frac{1}{\#(\Lambda_k^{\vec{v}}(b))}\sum_{(m,n)\in\Lambda_k^{\vec{v}}(b)}d(T^{(m,n)}x,T^{(m,n)}x_j)
+\frac{1}{\#(\Lambda_k^{\vec{v}}(b))}\sum_{(m,n)\in\Lambda_k^{\vec{v}}(b)}d(T^{(m,n)}x_j,T^{(m,n)}y)\\
<&\epsilon/2+\epsilon/2=\epsilon.
\end{align*}
So  for any $k\in \mathbb{N}$, there exists $x\in H$ such that  $\hat{d}^{\vec{v},b}_k(x,y)<\epsilon$. As $y\in K_0$ is arbitrary, $$K_0\subset \bigcup_{x\in H}B_{\hat{d}^{\vec{v},b}_k}(x,\epsilon).$$
This implies that for each $k\in\mathbb{N}$,
$$\mu\left(\bigcup_{x\in H}B_{\hat{d}^{\vec{v},b}_k}(x,\epsilon)\right)\geq \mu(K_0)>1-\epsilon.$$
Thus $\widehat{span}_{\mu}^{\vec{v},b}(k,\epsilon)\leq \#(H)$ for all $k\in \mathbb{N}$, which shows that $\mu$ has bounded complexity with respect to $\{\hat{d}^{\vec{v},b}_k\}$.
This finishes the proof of Theorem \ref{thm4}.
\end{proof}
 \end{thm}

 \subsection{Measure-theoretic complexity with respect to $\{\bar{d}_k^{\vec{v},b}\}$}Let $(X,T)$ be a  $\mathbb{Z}^2$-t.d.s. with a metric $d$ and $\mu\in M(X,T)$. Let   $\vec{v}=(1,\beta)\in\mathbb{R}^2$ be a direction vector and $b\in (0,\infty)$. For $k\in \mathbb{N}$ and $\epsilon>0$,  let
 $$\overline{span}_{\mu}^{\vec{v},b}(k,\epsilon)=\min\left\{\#(F):F\subset X\text{ and } \mu\left(\bigcup_{x\in F}B_{\bar{d}_k^{\vec{v},b}}(x,\epsilon)\right)>1-\epsilon\right\}.$$
 We say that $\mu$ has bounded complexity with respect to $\{\bar{d}_k^{\vec{v},b}\}$ if for any $\epsilon>0$, there exists a positive integer $C=C(\epsilon)$ such that $\overline{span}_{\mu}^{\vec{v},b}(k,\epsilon)\leq C$ for all $k\in \mathbb{N}$.

Now we begin to establish the relation between directional discrete spectrum and directional bounded complexity.
First we recall a $\mathbb{Z}$-m.p.s., which was introduced by Park \cite{P} as follows.
 Let $\widetilde{X}=X\times [0,1)^2$, $\widetilde{\mu}=\mu\times m$, $\widetilde{\mathcal{B}}=\mathcal{B}_X\times \mathcal{C}$, where $\mathcal{C}$ is the Borel  $\sigma$-algebra on $[0,1)^2$ and $m$ is the Lebesgue measure on $[0,1)^2$. Let $\phi_{s,t}(x,u,v)=(T^{([s+u],[t+v])}x,\{s+u\},\{t+v\})$, where $[a]$ is the integer part of $a$ and $\{a\}$ is the decimal part of $a$. Write $\phi_{n,n\beta}$ as $W^n$ for all $n\in\mathbb{Z}$. Then we get the $\mathbb{Z}$-m.p.s. $(\widetilde{X},\widetilde{\mathcal{B}},\widetilde{\mu},W).$ Let $\mathcal{K_{\widetilde{\mu}}}$ be the Kronecker algebra of $(\widetilde{X},\widetilde{\mathcal{B}},\widetilde{\mu},W).$ We can restate \cite[Lemma 3.4, Step 1]{LX} as follows.
 \begin{lem}\label{lem2}
$\widetilde{\mu}$ has discrete spectrum if and only if $\mu$  has $\vec{v}$-discrete spectrum.
 \end{lem}

We also restate \cite[Theorem 4.7]{H1} and \cite[Proposition 4.1]{HWY} as follows.
\begin{prop}\label{prop2}
Let $(X,T)$ be a $\mathbb{Z}$-t.d.s. with a metric $d$ and $\mu\in M(X,T)$. Then $\mu$ has discrete spectrum if and only if it has bounded complexity with respect to $\{\bar{d}_k\}$.
\end{prop}

With the help of Proposition \ref{prop2}, we will prove that for any $\mu\in M(X,T)$, having $\vec{v}$-discrete spectrum and having bounded complexity with respect to $\{\bar{d}_k^{\vec{v},b}\}$ for any $b\in (0,\infty)$ are equivalent.  The proof is divided into two steps, that is, Theorem \ref{thm5} and Theorem \ref{thm6}.
\begin{thm}\label{thm5}
Let $(X,T)$ be a  $\mathbb{Z}^2$-t.d.s., $\mu\in M(X,T)$ and $\vec{v}=(1,\beta)\in\mathbb{R}^2$ be a direction vector. If $\mu$ has $\vec{v}$-discrete spectrum, then it has bounded complexity with respect to $\{\bar{d}_k^{\vec{v},b}\}$ for any $b\in (0,\infty)$.
\end{thm}
\begin{proof}
Note that if $\beta\in \mathbb{Q}$ the statement can be proved by the same methods for the case of $\mathbb{Z}$-actions. So we assume that $\beta \notin \mathbb{Q}$.	
	
For the $\mathbb{Z}$-m.p.s. $(\widetilde{X},\widetilde{\mathcal{B}},\widetilde{\mu},W),$  we can define a metric $\tilde{d}$, which is compatible with the topology of $\widetilde{X}$, by
$$\tilde{d}((x,s,t),(y,u,v))=\max\{d(x,y),\min\{|s-u|,1-|s-u|\},\min\{|t-v|,1-|t-v|\}\}$$
 for all $(x,s,t),(y,u,v)\in \widetilde{X}$, where $|a|$ is absolute value of $a\in \mathbb{R}.$ As $\mu$ has $\vec{v}$-discrete spectrum, we know that $\widetilde{\mu}$ has discrete spectrum by Lemma \ref{lem2}. Using Proposition \ref{prop2}, we know that $\widetilde{\mu}$ has bounded complexity with respect to $\{\bar{\tilde{d}}_k\}$.
 
 Now we prove that $\mu$ has bounded complexity with respect to $\{\bar{d}_k^{\vec{v},b}\}$ for any $b\in (0,\infty)$ by the fact that $\widetilde{\mu}$ has bounded complexity with respect to $\{\bar{\tilde{d}}_k\}$. Fix $b\in(0,\infty)$.
  Since $X$ is compact, there is  $M\geq 1$ such that $
 diam(X)\leq M,
 $ where $diam(X)=\max\{d(x,y):x,y\in X\}$. As $T$ is continuous, for any $\epsilon\in(0,1)$, there exists  $\delta\in (0,\epsilon^2/(8M))$ such that for any $x,y\in X$ with $d(x,y)<\delta$, we have
\begin{equation}\label{4.1}
d(T^{(0,i)}x,T^{(0,i)}y)<\frac{b\epsilon^2}{16M}
\end{equation}
for each $i\in\{1,0,-1\}$. Since $\widetilde{\mu}$ has bounded complexity with respect to $\{\bar{\tilde{d}}_k\}$, there is a positive integer $C>0$ such that for any $k\in \mathbb{N}$, there exists $\widetilde F_k\subset X$ with $\#(\widetilde F_k)\leq C$ such that
$$\widetilde{\mu}\left(\bigcup_{(x,s,t)\in \widetilde F_k}B_{\bar{\tilde{d}}_k}\left((x,s,t),\frac{b^2\delta^2}{8M}\right)\right)>1-\frac{b^2\delta^2}{8M}.$$
Let $F_k=\{x\in X:\text{there exists } (s,t)\in [0,1)^2 \text{ such that }(x,s,t)\in \widetilde F_k\}$  for each  $k\in\mathbb{N}$. Then $\#(F_k)\leq\#(\widetilde{F}_k)\leq C$ for each $k\in\mathbb{N}$. Note that for any $x,y\in X$ and $(s,t)\in [0,1)^2$  we have
\begin{equation}\label{13}
\begin{split}
\bar{\tilde{d}}_k((x,s,t),(y,s,t))
&=\frac{1}{k}\sum_{i=0}^{k-1}\tilde{d}(W^i(x,s,t),W^i(y,s,t))\\
&\geq \frac{1}{k}\sum_{i=0}^{k-1}d(T^{(i,[i\beta+t])}x,T^{(i,[i\beta+t])}y).
\end{split}	
\end{equation}
 Then we have the following claim.
\begin{claim}\label{123}
Fix $x,y\in X$ and $(s,t)\in [0,1)^2$. Then for any $b\in(0,1]$ there exists $N>0$ such that for any $k\geq N$ if $\bar{\tilde{d}}_k((x,s,t),(y,s,t))\leq\frac{b^2\delta^2}{8M}$, then $$\frac{1}{\#(\Lambda_k^{\vec{v}}(b))}\sum_{(m,n)\in\Lambda_k^{\vec{v}}(b)}d(T^{(m,n)}x,T^{(m,n)}y)<\epsilon.$$
\end{claim} 
 \begin{proof}[Proof of Claim 1]
 Fix $b\in(0,1]$. Let $\mathcal{S}^b_t=\{i\in\mathbb{Z}_+:d(T^{(i,[i\beta+t])}x,T^{(i,[i\beta+t])}y)\geq b\delta\}$.
Combining \eqref{13} and the assumption $\bar{\tilde{d}}_k((x,s,t),(y,s,t))\leq\frac{b^2\delta^2}{8M}$,  we deduce that $$b\delta\frac{\#(\mathcal{S}_t^b\cap [0,k-1])}{\#([0,k-1])}\leq\frac{1}{k}\sum_{i=0}^{k-1}d(T^{(i,[i\beta+t])}x,T^{(i,[i\beta+t])}y)\le \bar{\tilde{d}}_k((x,s,t),(y,s,t))	\leq\frac{b^2\delta^2}{8M}. $$ That is,
 \begin{equation}\label{4.2}
 \frac{\#(\mathcal{S}_t^b\cap [0,k-1])}{\#([0,k-1])}\leq\frac{b\delta}{8M}.
 \end{equation} 
 Note that for each $t\in [0,1)$, $[i\beta+1-t]$ is equal to either $[i\beta+t]-1$, $[i\beta+t]$ or $[i\beta+t]+1$.
 Hence by \eqref{4.1}, for any $i\notin \mathcal{S}_t^b$, one has 
 \begin{equation}\label{14}
 d(T^{(i,[i\beta+1-t])}x,T^{(i,[i\beta+1-t])}y)<\frac{b\epsilon^2}{16M}.
 \end{equation}
 Then by \eqref{4.2}, \eqref{14} and the assumption that $\delta\in(0,\epsilon^2/(8M))$, one has
 \begin{equation}\label{15}
 \begin{split}
&\quad\frac{1}{\#(\Lambda_k^{\vec{v}}(1))}\sum_{(m,n)\in\Lambda_k^{\vec{v}}(1)}d(T^{(m,n)}x,T^{(m,n)}y)\\
&\leq \frac{1}{2k}\sum_{i=0}^{k-1}d(T^{(i,[i\beta+t])}x,T^{(i,[i\beta+t])}y)
+\frac{1}{2k}\sum_{i=0}^{k-1}d(T^{(i,[i\beta+1-t])}x,T^{(i,[i\beta+1-t])}y)\\
&\leq \frac{1}{2}\cdot\frac{\#(\mathcal{S}_t^b\cap [0,k-1])}{\#([0,k-1])}M+\frac{1}{2}b\delta+\frac{1}{2}\cdot\frac{\#(\mathcal{S}_t^b\cap [0,k-1])}{\#([0,k-1])}M+\frac{1}{2}\cdot\frac{b\epsilon^2}{16M}\\
&\leq  \frac{b\epsilon^2}{128M}+\frac{b\epsilon^2}{16M}+\frac{b\epsilon^2}{128M}+\frac{b\epsilon^2}{32M}< \frac{b\epsilon^2}{8M}.
 \end{split}
 \end{equation}

 Let $\mathcal{A}=\{(m,n)\in \Lambda^{\vec{v}}(1): d(T^{(m,n)}x,T^{(m,n)}y)\geq \epsilon/2\}$. Then by \eqref{15} we obtain that
 \begin{align*}
 \frac{\epsilon}{2}\cdot\frac{\#(\mathcal{A}\cap ([0,k-1]\times \mathbb{Z}))}{\#(\Lambda_k^{\vec{v}}(1))}
 \leq& \frac{1}{\#(\Lambda_k^{\vec{v}}(1))}\sum_{(m,n)\in\mathcal{A}\cap ([0,k-1]\times \mathbb{Z})}d(T^{(m,n)}x,T^{(m,n)}y)\\
 \leq &\frac{1}{\#(\Lambda_k^{\vec{v}}(1))}\sum_{(m,n)\in\Lambda_k^{\vec{v}}(1)}d(T^{(m,n)}x,T^{(m,n)}y)\leq \frac{b\epsilon^2}{8M}.
 \end{align*}
 Hence
 \begin{align}\label{2}
 \frac{\#(\mathcal{A}\cap ([0,k-1]\times \mathbb{Z}))}{\#(\Lambda_k^{\vec{v}}(1))}\leq\frac{b\epsilon}{4M}.
 \end{align}
 Since $\beta$ is an irrational number,  there exists $N>0$, for any $k\geq N$, we have
 \begin{align}\label{3}
 \frac{\#(\Lambda_k^{\vec{v}}(b))}{\#(\Lambda_k^{\vec{v}}(1))}>b/2.
 \end{align}
 Therefore, for any $k\geq N$, one has 
 \begin{equation*}
 \begin{split}
 &\frac{1}{\#(\Lambda_k^{\vec{v}}(b))}\sum_{(m,n)\in\Lambda_k^{\vec{v}}(b)}d(T^{(m,n)}x,T^{(m,n)}y)\\
 \leq& \frac{1}{\#(\Lambda_k^{\vec{v}}(b))}\sum_{(m,n)\in\Lambda_k^{\vec{v}}(b)\cap \mathcal{A}}d(T^{(m,n)}x,T^{(m,n)}y)+\epsilon/2\\
 \leq &M\frac{\#(\Lambda_k^{\vec{v}}(b)\cap \mathcal{A})}{\#(\Lambda_k^{\vec{v}}(b))}+\epsilon/2\\
 \overset{\eqref{3}}\leq& M\frac{\#(\mathcal{A}\cap ([0,k-1]\times \mathbb{Z}) )}{\#(\Lambda_k^{\vec{v}}(1))\cdot (b/2)}+\epsilon/2\\
 \overset{\eqref{2}}\leq &\frac{2M}{b}\cdot\frac{b\epsilon}{4M}+\epsilon/2=\epsilon.
 \end{split}
 \end{equation*}
 Now we finish the proof of Claim \ref{123}.
 \end{proof}
Next we finish our proof, which is divided into two cases, that is, $0<b\leq 1$ and $b>1$.
\begin{case}\label{C1}
$0<b\leq 1$.
\end{case}
If $k<N$, the consequence is obtained by the compactness of $X$ and the continuity of $T$. Now we consider $k\geq N$.
Note that there is  a fact that for any $y\in X$, if there is $(u,v)\in [0,1)^2$ such that $(y,u,v)\in B_{\bar{\tilde{d}}_k}\left((x,s,t),\frac{b^2\delta^2}{8M}\right)$ for some $(x,s,t)\in \widetilde F_k$, then $(y,s,t)\in B_{\bar{\tilde{d}}_k}((x,s,t),\frac{b^2\delta^2}{8M})$. By this fact and Claim \ref{123}, for any $y\in X$, if there is  $(u,v)\in [0,1)^2$ such that $(y,u,v)\in \bigcup_{(x,s,t)\in  \widetilde F_k}B_{\bar{\tilde{d}}_k}\left((x,s,t),\frac{b^2\delta^2}{8M}\right)$, then $y\in\bigcup_{x\in F_k}B_{\bar{d}^{\vec{v},b}_k}(x,\epsilon) $.
Let $$\mathcal{G}_k=\left\{y\in X:\text{for any }(u,v)\in [0,1)^2\text{, } (y,u,v)\notin \bigcup_{(x,s,t)\in  \widetilde F_k}B_{\bar{\tilde{d}}_k}\left((x,s,t),\frac{b^2\delta^2}{8M}\right)\right\}.$$
Then $\left(\mathcal{G}_k\times[0,1)^2\right)\bigcap\left(\bigcup_{(x,s,t)\in  \widetilde F_k}B_{\bar{\tilde{d}}_k}\left((x,s,t),\frac{b^2\delta^2}{8M}\right)\right)=\emptyset.$  So $\widetilde{\mu}\left(\mathcal{G}_k\times [0,1)^2\right)<1-\widetilde{\mu}\left(\bigcup_{(x,s,t)\in  \widetilde F_k}B_{\bar{\tilde{d}}_k}\left((x,s,t),\frac{b^2\delta^2}{8M}\right)\right)<\frac{b^2\delta^2}{4M}\leq \frac{b^2}{4M}\cdot \frac{\epsilon^4}{4M^2}<\epsilon$, that is, $\mu\left(\mathcal{G}_k\right)<\epsilon$. Hence $\mu\left(\bigcup_{x\in F_k}B_{\bar{d}^{\vec{v},b}_k}(x,\epsilon)\right)\ge1-\mu\left(\mathcal{G}_k\right)>1-\epsilon$. Combined with the fact that $\#(F_k)\leq C$ for each $k\in\mathbb{N}$, $\mu$ has bounded complexity with respect to $\{\bar{d}_k^{\vec{v},b}\}$.
\begin{case}\label{C2}
$b>1$.
\end{case}
 Note that we can find a finite subset $C$ of $\mathbb{Z}^2$ such that $$\Lambda^{\vec{v}}(b)\subset\bigcup_{(m',n')\in C}\left((m',n') +\Lambda^{\vec{v}}(1)\right)$$
 where $(m',n') +\Lambda^{\vec{v}}(1)=\{(m'+s,n'+t):(s,t)\in \Lambda^{\vec{v}}(1)\}$.  Since  $\mu$ has bounded complexity with respect to $\{\bar{d}_k^{\vec{v},1}\}$, $X$ is compact and $T$ is continuous, it follows  that $\mu$ has bounded complexity with respect to $\{\bar{d}_k^{\vec{v},b}\}$.

Now we have proved $\mu$ has bounded complexity with respect to $\{\bar{d}_k^{\vec{v},b}\}$ for any $b\in (0,\infty)$.
\end{proof}

The following result can be obtained immediately from the proof of Theorem \ref{thm5}.
\begin{cor}\label{cor1}
Let $(X,T)$ be a  $\mathbb{Z}^2$-t.d.s., $\mu\in M(X,T)$ and  $\vec{v}=(1,\beta)\in\mathbb{R}^2$ be a direction vector. Then the following two statements are equivalent.
\begin{itemize}
 \item[(a)]For any $b\in (0,\infty)$, $\mu$ has bounded complexity with respect to $\{\bar{d}_k^{\vec{v},b}\}$.
\item[(b)]For some $b\in (0,\infty)$, $\mu$ has bounded complexity with respect to $\{\bar{d}_k^{\vec{v},b}\}$.
\end{itemize}
\end{cor}

Now we show the converse of Theorem \ref{thm5}.
 \begin{thm}\label{thm6}
 	Let $(X,T)$ be a  $\mathbb{Z}^2$-t.d.s., $\mu\in M(X,T)$ and  $\vec{v}=(1,\beta)\in\mathbb{R}^2$ be a direction vector. If $\mu$ has bounded complexity with respect to $\{\bar{d}_k^{\vec{v},b}\}$ for some $b\in (0,\infty)$  then it  has $\vec{v}$-discrete spectrum.
 \end{thm}
 \begin{proof}
 	In the process of proof we still use the notation in the proof of Theorem \ref{thm5} and also only consider the case $\beta\notin\mathbb{Q}$.

 By Corollary \ref{cor1}, we may assume that $b=1$. Note that for any $t\in [0,1)$ and $i\in\mathbb{N}$, $(i,[i\beta+t])\in \Lambda^{\vec{v}}(1)$. 
As $\mu$ has bounded complexity with respect to $\{\bar{d}_k^{\vec{v},1}\}$, we know that there exists  $C=C(\epsilon)>0$ such that for any $k\in \mathbb{N}$, there exists $F_k\subset X$ with $\#(F_k)\leq C$ such that
$$\mu\left(\bigcup_{x\in F_k}B_{\bar{d}^{\vec{v},1}_k}\left(x,\frac{\epsilon^2}{32M^2}\right)\right)>1-\frac{\epsilon^2}{32M^2}.$$
Since the aciton $W$ on the second coordinate is indetity,  and the third coordinate is an irrational rotation of the circle, there exists $C_1>0$ such that for any $k\in\mathbb{N}$, there exists $E_k\subset [0,1)$ with $\#(E_k)\leq C_1$ such that
$$m\left(\bigcup_{s\in E_k}B_{\rho_k}(s,\epsilon/(8M))\times \bigcup_{t\in E_k}B_{\rho_k}(t,\epsilon/(8M))\right)=1,$$
where $$\rho_k(s,u)=\max_{0\leq i\leq k-1}\left\{\min\{|\{i\beta+s\}-\{i\beta+u\}|,1-|\{i\beta+s\}-\{i\beta+u\}|\}\right\}.$$ Let $\widetilde{F}_k=\{(x,s,t)\in \widetilde{X}: x\in F_k, (s,t)\in E_k\times E_k\}$. Since $F_k$ and $E_k$ are finite set for each $k\in \mathbb{N}$, there exists $\widetilde{C}>0$ such that $\#(\widetilde{F}_k)\leq \widetilde{C}$ for each $k\in \mathbb{N}$. By Birkhoff ergodic theory,  for any $s,u \in [0,1)$, if $|u-s|<\epsilon/(8M)$, then there exists  $N_1>0$, whenever $k\geq N_1$, we have
\begin{equation}\label{4}
\frac{1}{k}\#(\{i\in[0,k-1]:[i\beta+u]\neq[i\beta+s]\})<\epsilon/(4M).
\end{equation}

Now we prove that $\widetilde{\mu}$ has bounded complexity with repsect to $\{\bar{\tilde{d}}_k\}$ by proving that $\widetilde{\mu}\left(\bigcup_{(x,s,t)\in \widetilde F_k}B_{\bar{\tilde{d}}_k}\left((x,s,t),\epsilon\right)\right)\geq1-\epsilon$.
For any $(y,u,v)\in\bigcup_{x\in F_k}B_{\bar{d}^{\vec{v},1}_k}(x,\frac{\epsilon^2}{32M^2})\times \bigcup_{s\in E_k}B_{\rho_k}(s,\epsilon/8M)\times \bigcup_{t\in E_k}B_{\rho_k}(t,\epsilon/8M)$, there exist $x\in F_k$ and $(s,t)\in E_k\times E_k$ such that $y\in B_{\bar{d}^{\vec{v},b}_k}(x,\frac{\epsilon^2}{32M^2})$ and $(u,v)\in B_{\rho_k}(s,\epsilon/(8M))\times B_{\rho_k}(t,\epsilon/(8M))$.
If $k\geq N_1$, then
\begin{equation}\label{6}
\begin{split}
&\bar{\tilde{d}}_k((x,s,t),(y,u,v))=\frac{1}{k}\sum_{i=0}^{k-1}\tilde{d}(W^i(x,s,t),W^i(y,u,v))\\
\leq&\frac{1}{k}\sum_{i=0}^{k-1}\max\left\{d(T^{(i,[i\beta+t])}x,T^{(i,[i\beta+v])}y),\frac{\epsilon}{4M}\right\}\\
\leq&\frac{1}{k}\sum_{i=0}^{k-1}\max\left\{d(T^{(i,[i\beta+t])}x,T^{(i,[i\beta+t])}y)+d(T^{(i,[i\beta+t])}y,T^{(i,[i\beta+v])}y),\frac{\epsilon}{4M}\right\}\\
\leq&\frac{1}{k}\sum_{i=0}^{k-1}d(T^{(i,[i\beta+t])}y,T^{(i,[i\beta+v])}y)+\frac{1}{k}\sum_{i=0}^{k-1}\max\left\{d(T^{(i,[i\beta+t])}x,T^{(i,[i\beta+t])}y),\frac{\epsilon}{4M}\right\}\\
\overset{\eqref{4}}<&\epsilon/2+\frac{1}{k}\sum_{i=0}^{k-1}\max\left\{d(T^{(i,[i\beta+t])}x,T^{(i,[i\beta+t])}y),\frac{\epsilon}{4M}\right\}.
\end{split}
\end{equation}
Let $\mathcal{A}=\{i\in \mathbb{Z}_+:d(T^{(i,[i\beta+t])}x,T^{(i,[i\beta+t])}y)>\epsilon/(4M)\}$. Then
\begin{align*}
(\epsilon/(4M))\cdot\frac{\#(\mathcal{A}\cap [0,k-1])}{2k+1}
\leq& \frac{1}{2k+1}\sum_{(m,n)\in\Lambda_k^{\vec{v}}(1)}d(T^{(m,n)}x,T^{(m,n)}y)\\
=&\frac{1}{\#(\Lambda_k^{\vec{v}}(1))}\sum_{(m,n)\in\Lambda_k^{\vec{v}}(1)}d(T^{(m,n)}x,T^{(m,n)}y)<\frac{\epsilon^2}{32M^2}.
\end{align*}
So there exists  $N_2>0$ such that for any $k\geq N_2$, we have $$\frac{\#(\mathcal{A}\cap [0,k-1])}{k}\le\frac{\epsilon}{4M}.$$ Let $N=\max\{N_1,N_2\}$. By \eqref{6}, we know that if $k\geq N$,
\begin{equation*}
\begin{split}
\bar{\tilde{d}}_k((x,s,t),(y,u,v))<& \epsilon/2+\frac{1}{k}\#(\mathcal{A}\cap [0,k-1])\cdot M+\frac{\epsilon}{4M} \\
\leq&\epsilon/2+M\cdot\frac{\epsilon}{4M}+\frac{\epsilon}{4M}\leq \epsilon.
\end{split}
\end{equation*}
It follows that for any $k\geq N$, $(y,u,v)\in B_{\bar{\tilde{d}}_k}\left((x,s,t),\epsilon\right)$.
Hence for any $k\geq N$, $$\bigcup_{x\in F_k}B_{\bar{d}^{\vec{v},1}_k}(x,\frac{\epsilon^2}{32M^2})\times \bigcup_{s\in E_k}B_{\rho_k}(s,\frac{\epsilon}{8M})\times \bigcup_{t\in E_k}B_{\rho_k}(t,\frac{\epsilon}{8M})\subset\bigcup_{(x,s,t)\in \widetilde F_k}B_{\bar{\tilde{d}}_k}\left((x,s,t),\epsilon\right).$$ So we have
\begin{equation*}
\begin{split}
&\widetilde{\mu}\left(\bigcup_{(x,s,t)\in \widetilde F_k}B_{\bar{\tilde{d}}_k}\left((x,s,t),\epsilon\right)\right)\\
\geq& \mu\left(\bigcup_{x\in F_k}B_{\bar{d}^{\vec{v},1}_k}\left(x,\frac{\epsilon^2}{32M^2}\right)\right)\cdot m\left(\bigcup_{s\in E_k}B_{\rho_k}(s,\epsilon/(8M))\times \bigcup_{t\in E_k}B_{\rho_k}(t,\epsilon/(8M))\right)\\
\geq& 1-\frac{\epsilon^2}{16M^2}>1-\epsilon.
\end{split}
\end{equation*}
If $k<N$, then the consequence is obtained by the compactness of $X$ and the continuity of $T$.
Thus we have shown $\widetilde{\mu}$ has bounded complexity with respect to $\{\bar{\tilde{d}}_k\}$. It follows from Proposition \ref{prop2} that $\widetilde{\mu}$ has discrete spectrum. Directly from Lemma \ref{lem2}, we know that $\mu$ has $\vec{v}$-discrete spectrum. This completes the proof of Theorem \ref{thm6}.
 \end{proof}

Then we directly obtain the following result by Theorem \ref{thm5} and Theorem \ref{thm6}.
 \begin{thm}\label{thm7}
 	Let $(X,T)$ be a  $\mathbb{Z}^2$-t.d.s., $\mu\in M(X,T)$ and $\vec{v}=(1,\beta)\in\mathbb{R}^2$ be a direction vector. Then the following statements are equivalent.
 	\begin{itemize}
 	\item[(a)]$\mu$ has bounded complexity with respect to $\{\bar{d}_k^{\vec{v},b}\}$ for any $b\in (0,\infty)$.
 	\item[(b)]$\mu$ has bounded complexity with respect to $\{\bar{d}_k^{\vec{v},b}\}$ for some $b\in (0,\infty)$.
 	\item[(c)]$\mu$ has $\vec{v}$-discrete spectrum.
 	\end{itemize}
 \end{thm}

Immediately from the proof of Theorem \ref{thm5} and Theorem \ref{thm6}, we can get the following result.
\begin{cor}\label{cor2}
	The following statements are equivalent.
	\begin{itemize}
		\item[(a)] $\widetilde{\mu}$ has bounded complexity with respect to $\{\bar{\tilde{d}}_k\}$. 	
	\item[(b)] $\mu$ has bounded complexity with respect to $\{\bar{d}_k^{\vec{v},b}\}$ for any $b\in (0,\infty)$.
	\item[(c)] $\mu$ has bounded complexity with respect to $\{\bar{d}_k^{\vec{v},b}\}$ for some $b\in (0,\infty)$.
	\end{itemize}
\end{cor}

It is different from the topological case that we will prove that bounded meausre-theoretic complexity with respect to $\{\bar{d}_k^{\vec{v},b}\}$ and $\{\hat{d}_k^{\vec{v},b}\}$ are equivalent for any $b\in (0,\infty)$. The corresponding result for $\mathbb{Z}$-t.d.s. has been shown by Huang et al.  \cite[Theorem 4.4]{H1}, which is restated as follows.
\begin{lem}\label{lem3}
Let $(X,T)$ be a $\mathbb{Z}$-t.d.s. with a metric $d$ and $\mu\in M(X,T)$. Then $\mu$ has bounded complexity with respect to $\{\overline {d}_k\}$ if and only if it has bounded complexity with respect to  $\{\hat{d}_k\}$.
\end{lem}
Taking advantage of Lemma \ref{lem3}, we obtain the equivalence of  bounded meausre-theoretic complexity with respect to $\{\bar{d}_k^{\vec{v},b}\}$ and $\{\hat{d}_k^{\vec{v},b}\}$.
 \begin{thm}\label{thm8}
Let $(X,T)$ be a  $\mathbb{Z}^2$-t.d.s., $\mu\in M(X,T)$ and  $\vec{v}=(1,\beta)\in\mathbb{R}^2$ be a direction vector. Then for any $b\in(0,\infty)$ the following statements are equivalent.
\begin{itemize}
\item[(a)]$\mu$ has bounded complexity with respect to $\{\bar{d}_k^{\vec{v},b}\}$.
\item[(b)]$\mu$ has bounded complexity with respect to $\{\hat{d}_k^{\vec{v},b}\}$.
\end{itemize}
\begin{proof}
(b) $\Rightarrow$ (a). Directly from definitions.

(a) $\Rightarrow$ (b). Suppose that $\mu$ has bounded complexity with respect to $\{\bar{d}_k^{\vec{v},b}\}$. By Corollary \ref{cor2}, we know that $\widetilde{\mu}$ has bounded complexity with respect to $\{\bar{\tilde{d}}_k\}$. So we obtain that $\widetilde{\mu}$ has bounded complexity with respect to $\{\hat{\tilde{d}}_k\}$ by Lemma \ref{lem3}.
Since $\hat{\tilde{d}}_k(x,y)=\max_{1\leq i\leq k}\{\bar{\tilde{d}}_i(x,y)\}$ and $\hat{d}_k^{\vec{v},b}(x,y)=\max_{1\leq i\leq k}\{\bar{d}_i^{\vec{v},b}(x,y)\}$ for each $k\in \mathbb{N}$ and $x,y\in X$, we may deduce that $\mu$ has bounded complexity with respect to $\{\hat{d}_k^{\vec{v},b}\}$ by the same method in the proof of Theorem \ref{thm5}. This finishes the proof of Theorem \ref{thm8}.
\end{proof}
 \end{thm}

\section*{Acknowledgements}
  The authors would like to thank Prof. Wen Huang for his useful comments and suggestions. C. Liu was partially supported by NNSF of China (12090012). L. Xu was partially supported by NNSF of China (11801538, 11871188, 12031019) and the USTC Research Funds of the Double First-Class Initiative.

\begin{appendix}
\section{Results for $\mathbb{Z}^q$-t.d.s.}
In this section, we introduce the corresponding results of $\mathbb{Z}^q$-t.d.s. which are proved by exactly the same methods for the case of $\mathbb{Z}^2$-t.d.s.  Let $(X ,T)$ be a $\mathbb{Z}^q$-t.d.s.  with a metric $d$ and $\mu\in M(X,T)$. Let $\vec{v}=(1,\beta_2,\ldots,\beta_q)\in\mathbb{R}^q$ be a direction vector and $\textbf{b}=(b_2,\ldots,b_q)\in \mathbb{R}^{q-1}_+:=\{\textbf{u}=(u_1,\ldots,u_{q-1})\in \mathbb{R}^{q-1}:u_i>0,\text{ }i=1,2\ldots,q-1\}$. We put $$\Lambda^{\vec{v}}(\textbf{b})=\{\vec{w}=(m_1,\ldots,m_q)\in\mathbb{Z}^q:\beta_i m_1-b_i\leq m_i\leq \beta_i m_1+b_i,\text{ } i=2,3,\ldots,q\}$$ 
and let $\Lambda_k^{\vec{v}}(\textbf{b})=\Lambda^{\vec{v}}(\textbf{b})\cap ([0,k-1]\times\mathbb{Z}^{q-1})$ for each $k\in\mathbb{N}$.
Define the $\vec{v}$-directional Kronecker algebra by
$$\mathcal{K}_\mu^{\vec{v}}(\textbf{b})=\left\{B\in\mathcal{B}_X: \overline{\{U_T^{\vec{w}}1_B :\vec{w}\in \Lambda^{\vec{v}}(\textbf{b}) \}}\text{ is compact in } L^2(X,\mathcal{B}_X,\mu) \right\},$$
where $U_T^{\vec{w}}:L^2(X,\mathcal{B}_X,\mu)\to L^2(X,\mathcal{B}_X,\mu)$ is the unitary operator such that
$U_T^{\vec{w}}f=f\circ T^{\vec{w}}\text{ for all }f\in L^2(X,\mathcal{B}_X,\mu).$
Similarly, we can prove that $\mathcal{K}_\mu^{\vec{v}}(\textbf{b})$ is a $\sigma$-algebra and the definition of $\mathcal{K}_\mu^{\vec{v}}(\textbf{b})$ is independent of the selection of $\textbf{b}$. So we omit $\textbf{b}$ in $\mathcal{K}_\mu^{\vec{v}}(\textbf{b})$ and write it as $\mathcal{K}_\mu^{\vec{v}}$. We say $\mu$ has $\vec{v}$-discrete spectrum system if $\mathcal{K}_\mu^{\vec{v}}=\mathcal{B}_X$.
	
Similar to the case of $q=2$, we define three kinds of metrics on $X$ along $\vec{v}$ as follows. For each $k\in \mathbb{N}$ and $x,y\in X$, let
$$d_k^{\vec{v},\textbf{b}}(x,y)=\max_{\vec{w}\in \Lambda_k^{\vec{v}}(\textbf{b})}\{d(T^{\vec{w}}x,T^{\vec{w}}y)\},$$

$$\hat{d}_k^{\vec{v},\textbf{b}}(x,y)=\max_{1\leq i\leq k}\left\{\frac{1}{\#(\Lambda_i^{\vec{v}}(\textbf{b}))}\sum_{\vec{w}\in\Lambda_i^{\vec{v}}(\textbf{b})}d(T^{\vec{w}}x,T^{\vec{w}}y)\right\}$$
and
$$\bar{d}_k^{\vec{v},\textbf{b}}(x,y)=\frac{1}{\#(\Lambda_k^{\vec{v}}(\textbf{b}))}\sum_{\vec{w}\in\Lambda_k^{\vec{v}}(\textbf{b})}d(T^{\vec{w}}x,T^{\vec{w}}y).$$

   Let $span_K^{\vec{v},\textbf{b}}(k,\epsilon)$ denote the smallest cardinality of any $(k,\epsilon)$-spanning set for $K$ with respect to $T$, that is  $$span^{\vec{v},\textbf{b}}_K(k,\epsilon)=\min\left\{\#(F):F\subset X\subset \bigcup_{x\in F}B_{d_k^{\vec{v},\textbf{b}}}(x,\epsilon) \right\}.$$
We say that a subset $K$ of $X$ has bounded topological complexity with respect to $\{d_k^{\vec{v},\textbf{b}}\}$ if for any $\epsilon>0$, there exists a positive integer $C=C(\epsilon)$ such that $span_K^{\vec{v},\textbf{b}}(k,\epsilon)\leq C$ for all $k\in \mathbb{N}$. And we say that a subset $K$ of $X$ is $(\vec{v},\textbf{b})$-equicontinuous if for any $\epsilon>0$, there is $\delta>0$ such that whenever $x,y\in K$ with $d(x,y)<\delta$, $d(T^{\vec{w}}x,T^{\vec{w}}y)<\epsilon$ for all $\vec{w}\in \Lambda^{\vec{v}}(\textbf{b})$.
Then we have the following consequence.
\begin{thm}
Let $(X,T)$ be a $\mathbb{Z}^q$-t.d.s., $K$ be a compact subset of $X$ and $\vec{v}=(1,\beta_2,\ldots,\beta_q)\in\mathbb{R}^q$ be a direction vector. Then for any  $\textbf{b}\in \mathbb{R}^{q-1}_+$, the following statements are equivalent.
 \begin{itemize}
	\item[(a)]$K$ has bounded topological complexity with respect to $\{d_k^{\vec{v},\textbf{b}}\}$ .
	\item[(b)]$K$ is $(\vec{v},\textbf{b})$-equicontinuous.
\end{itemize}
\end{thm}
Define
$$\widehat{span}_K^{\vec{v},\textbf{b}}(k,\epsilon)=\min\left\{\#(F):F\subset X\subset \bigcup_{x\in F}B_{\hat{d}_k^{\vec{v},\textbf{b}}}(x,\epsilon) \right\}.$$
We say that a subset $K$ of $X$ has bounded topological complexity with respect to $\{\hat{d}_k^{\vec{v},\textbf{b}}\}$ if for any $\epsilon>0$, there exists a positive integer $C=C(\epsilon)$ such that $\widehat{span}_K^{\vec{v},\textbf{b}}(k,\epsilon)\leq C$ for all $k\in \mathbb{N}$. Since $\hat{d}_k^{\vec{v},\textbf{b}}(x,y)\leq d_k^{\vec{v},\textbf{b}}(x,y)$ for all $k\in \mathbb{N}$ and $x,y\in X$, if $K$ has bounded topological complexity with respect to $\{d_k^{\vec{v},\textbf{b}}\}$, then it also has bounded topological complexity with respect to $\{\hat{d}_k^{\vec{v},\textbf{b}}\}$. We say that a subset $K$ of $X$ is $(\vec{v},\textbf{b})$-equicontinuous in the mean if for any $\epsilon>0$, there is $\delta>0$ such that $\hat{d}_k^{\vec{v},\textbf{b}}(x,y)<\epsilon$ for all $k\in \mathbb{N}$ and $x,y\in X$ with $d(x,y)<\delta$. Similarly, we have the following consequence.

\begin{thm}
Let $(X,T)$ be a $\mathbb{Z}^q$-t.d.s., $K$ be a compact subset of $X$ and $\vec{v}=(1,\beta_2,\ldots,\beta_q)\in\mathbb{R}^q$ be a direction vector. Then for any  $\textbf{b}\in \mathbb{R}^{q-1}_+$, the following statements are equivalent.
	\begin{itemize}
		\item[(a)]$K$ has bounded topological complexity with respect to $\{\hat{d}_k^{\vec{v},\textbf{b}}\}$.
		\item[(b)]$K$ is $(\vec{v},\textbf{b})$-equicontinuous in the mean.
	\end{itemize}
\end{thm}

Let
$$span_{\mu}^{\vec{v},\textbf{b}}(k,\epsilon)=\min\left\{\#(F):F\subset X\text{ and } \mu\left(\bigcup_{x\in F}B_{d_k^{\vec{v},\textbf{b}}}(x,\epsilon)\right)>1-\epsilon\right\}.$$
We say that $\mu$ has bounded complexity with respect to $\{d_k^{\vec{v},\textbf{b}}\}$ if for any $\epsilon>0$, there exists a positive integer $C=C(\epsilon)$ such that $span_{\mu}^{\vec{v},\textbf{b}}(k,\epsilon)\leq C$ for all $k\in \mathbb{N}$. We say that $T$ is $(\mu,\vec{v},\textbf{b})$-equicontinuous if for any $\tau>0$, there exists a  measurable $(\vec{v},\textbf{b})$-equicontinuous subset $K$ of $X$ with $\mu(K)>1-\tau$.
\begin{thm}
	Let $(X,T)$ be a $\mathbb{Z}^q$-t.d.s., $\mu\in M(X,T)$ and   $\vec{v}=(1,\beta_2,\ldots,\beta_q)\in\mathbb{R}^q$ be a direction vector. Then for any  $\textbf{b}\in \mathbb{R}^{q-1}_+$, the following two statements are equivalent.
	\begin{itemize}
	\item[(a)]$\mu$ has bounded complexity with respect to $\{d_k^{\vec{v},\textbf{b}}\}$.
    \item[(b)]$T$ is $(\mu,\vec{v},\textbf{b})$-equicontinuous.
	\end{itemize}

\end{thm}
Let
$$\widehat{span}_{\mu}^{\vec{v},\textbf{b}}(k,\epsilon)=\min\left\{\#(F):F\subset X\text{ and } \mu\left(\bigcup_{x\in F}B_{\hat{d}_k^{\vec{v},\textbf{b}}}(x,\epsilon)\right)>1-\epsilon\right\}$$
and
$$\overline{span}_{\mu}^{\vec{v},\textbf{b}}(k,\epsilon)=\min\left\{\#(F):F\subset X\text{ and } \mu\left(\bigcup_{x\in F}B_{\bar{d}_k^{\vec{v},\textbf{b}}}(x,\epsilon)\right)>1-\epsilon\right\}.$$

We say that $\mu$ has bounded complexity with respect to $\{\hat{d}_k^{\vec{v},\textbf{b}}\}$ if for any $\epsilon>0$, there exists a positive integer $C=C(\epsilon)$ such that $\widehat{span}_{\mu}^{\vec{v},\textbf{b}}(k,\epsilon)\leq C$ for all $k\in \mathbb{N}$. We say that $T$ is $(\mu,\vec{v},\textbf{b})$-equicontinuous in the mean if for any $\tau>0$, there exists a measurable subset $K$ of $X$ with $\mu(K)>1-\tau$ that is $(\vec{v},\textbf{b})$-equicontinuous in the mean and $(\mu,\vec{v},\textbf{b})$-mean equicontinuous, if for any $\tau>0$, there exists a measurable subset $K$ of $X$ with $\mu(K)>1-\tau$ that is $(\vec{v},\textbf{b})$-mean equicontinuous.

We say that $\mu$ has bounded complexity with respect to $\{\bar{d}_k^{\vec{v},\textbf{b}}\}$ if for any $\epsilon>0$, there exists a positive integer $C=C(\epsilon)$ such that $\overline{span}_{\mu}^{\vec{v},\textbf{b}}(k,\epsilon)\leq C$ for all $k\in \mathbb{N}$.  Using exactly the same methods as proving the case of $\mathbb{Z}^2$-t.d.s., we have the corresponding consequence as follows.
\begin{thm}
Let $(X,T)$ be a $\mathbb{Z}^q$-t.d.s., $\mu\in M(X,T)$ and  $\vec{v}=(1,\beta_2,\ldots,\beta_q)\in\mathbb{R}^q$ be a direction vector. Then the following statements are equivalent.
\begin{itemize}
\item[(a)]$\mu$ has bounded complexity with respect to $\{\hat{d}_k^{\vec{v},\textbf{b}}\}$ for some $\textbf{b}\in \mathbb{R}^{q-1}_+$ $($or for any $\textbf{b}\in \mathbb{R}^{q-1}_+)$.
\item[(c)]$\mu$ has bounded complexity with respect to $\{\bar{d}_k^{\vec{v},\textbf{b}}\}$ for some $\textbf{b}\in \mathbb{R}^{q-1}_+$ $($or for any $\textbf{b}\in \mathbb{R}^{q-1}_+)$.
\item[(e)]$\mu$ has $\vec{v}$-discrete spectrum.
\item[(f)]$T$ is $(\mu,\vec{v},\textbf{b})$-equicontinuous in the mean for some $\textbf{b}\in \mathbb{R}^{q-1}_+$ $($or for any $\textbf{b}\in \mathbb{R}^{q-1}_+)$.
\item[(h)]$T$ is $(\mu,\vec{v},\textbf{b})$-mean equicontinuous for some $\textbf{b}\in \mathbb{R}^{q-1}_+$ $($or for any $\textbf{b}\in \mathbb{R}^{q-1}_+)$.
\end{itemize}

\end{thm}

\end{appendix}

\end{document}